\documentclass[preprint,3p,times]{elsarticle}
\usepackage[toc,page,title,titletoc,header]{appendix}
\usepackage{amsmath}
\usepackage{amssymb}
\usepackage{amsthm}
\usepackage{enumerate}
\usepackage{enumitem}
\usepackage{mathrsfs}
\usepackage{graphicx}
\usepackage{epsfig}
\usepackage{tikz,pgfplots,tikz-3dplot}
\usepackage{epstopdf}
\usepackage{microtype}
\usetikzlibrary{fit}
\biboptions{sort&compress}

\newtheorem{thm}{Theorem}[section]
\newtheorem{lem}[thm]{Lemma}
\newtheorem{rem}{Remark}
\renewcommand{}

\newcommand{\bI}{\mathbb{I}}

\newcommand{\nn}{\nonumber}

\newcommand{\ttau}{\Delta t}

\def\epsilon{\varepsilon} 
\newcommand{\mat}[1]{\boldsymbol{#1}}

\allowdisplaybreaks
\begin{document}
\begin{frontmatter}
\title{
Energy-stable parametric finite element approximations for regularized solid-state dewetting in strongly
anisotropic materials
}

\author[1]{Meng Li}
\address[1]{School of Mathematics and Statistics, Zhengzhou University,
Zhengzhou 450001, China.}
\ead{This author's research was supported by National Natural Science Foundation of China (No. 11801527), the China Postdoctoral Science Foundation (No. 2023T160589).
Corresponding author: limeng@zzu.edu.cn. }
\author[1]{Chunjie Zhou}

\begin{abstract}
In this work, we aim to develop energy-stable parametric finite element approximations for a sharp-interface model with strong surface energy anisotropy, which is derived from the first variation of an energy functional composed of film/vapor interfacial energy, substrate energy, and regularized Willmore energy. By introducing two geometric relations, we innovatively establish an equivalent regularized sharp-interface model and further construct an energy-stable parametric finite element algorithm for this equivalent model. We provide a detailed proof of the energy stability of the numerical scheme, addressing a gap in the relevant theory. Additionally, we develop another structure-preserving parametric finite element scheme that can preserve both area conservation and energy stability. Finally, we present several numerical simulations to show accuracy and efficiency as well as some structure-preserving properties of the proposed numerical methods. More importantly, extensive numerical simulations reveal that our schemes provide better mesh quality and are more suitable for long-term computations. 
\end{abstract}




\end{frontmatter}

\pagestyle{myheadings} \markboth{~}
{}

\section{Introduction}

Solid-state dewetting is a widely observed phenomenon in physics and materials science.
It occurs in solid–solid
–vapor systems, and could be used to describe the agglomerative process of solid thin films on a substrate. 
Due to the effects of surface tension and capillarity, a solid film adhered to the substrate is inherently unstable or metastable in its as-deposited state. This instability can lead to complex morphological evolutions, such as edge retraction \cite{Wong00,dornel2006surface,hyun2013quantitative}, faceting \cite{Jiran90,Jiran92,Ye10a}, pinch-off events \cite{Jiang12,Kim15} and fingering instabilities \cite{Kan05,Ye10b,Ye11a,Ye11b}. 
The phenomenon, known as solid-state dewetting, is named as such because the thin film remains solid throughout the process \cite{Thompson12}. 
Recently, solid-state dewetting has found widespread applications in various modern technologies, see \cite{Mizsei93,Armelao06,Schmidt09} and the references therein. The significant and diverse applications of solid-state dewetting have sparked growing interest and efforts to comprehend its underlying mechanism, including the experimental investigations in \cite{Jiran90,Jiran92,Ye10a,Ye10b,Ye11a,Amram12,Rabkin14,Herz216,Naffouti16,Naffouti17,Kovalenko17}, as well as theoretical studies documented in \cite{Wong00,Dornel06,hyun2013quantitative,Jiang12,Jiang16,Kim15,Kan05,Srolovitz86a,Srolovitz86,Wang15,baojcm2022,Bao17,Bao17b,Zucker16}.

Several solid-state dewetting models have been proposed for isotropic surface energy cases \cite{Srolovitz86,Wong00,Dornel06,Jiang12,Zhao20}. However, the kinetic evolution during solid-state dewetting is significantly affected by crystalline anisotropy, as demonstrated by numerous experiments \cite{Thompson12,Leroy16}.
In recent years, various approaches have been explored to theoretically investigate the effects of surface energy anisotropy on the solid-state dewetting (refer to \cite{Dornel06,Pierre09b,Dufay11,Klinger11shape,Zucker13,Bao17,Jiang16,Wang15,Jiang19a,Zhao19b} and the references therein).

We first review the sharp-interface model for the solid-state dewetting with isotropic/weakly anisotropic surface energies. Let $\Gamma:=\Gamma(t)$ be an open curve with two triple (or contact) points $x_c^l:=x_c^l(t)$ and $x_c^r:=x_c^r(t)$. 
Denote $\vec{n}=(-\sin\theta, \cos\theta)^T$ as the unit outward normal vector, where $\theta \in [-\pi, \pi]$ is the angle between the unit outward normal and vertical axis. Let $\vec{\tau} = \vec{n}^\perp$ be the unit tangential vector, and $\vec{n}^\perp$ denotes the clockwise rotation of $\vec{n}$ by $90$ degrees.
Consider the total interfacial energy 
\begin{equation}
 \label{eq:InterEg1}
W(\Gamma)=\int_{\Gamma_{FV}}\gamma_{FV}d\Gamma_{FV}+\int_{\Gamma_{FS}}\gamma_{FS}d\Gamma_{FS}+\int_{\Gamma_{VS}}\gamma_{VS}d\Gamma_{VS}, 
\end{equation}
where $\Gamma_{FV}:=\Gamma$, $\Gamma_{FS}$, and $\Gamma_{VS}$ denote the film/vapor, film/substrate, and vapor/substrate interfaces, respectively, with corresponding surface energy densities $\gamma_{FV}$, $\gamma_{FS}$, and $\gamma_{VS}$. 
In the context of solid-state dewetting  problems, the surface energy densities $\gamma_{FS}$ and $\gamma_{VS}$ are typically assumed to be constant values. On the other hand, $\gamma_{FV}=\gamma(\theta)\in C^2([-\pi, \pi])$ depends on the orientation of the film/vapor interface. Consequently, the total interfacial energy \eqref{eq:InterEg1} can be expressed in the following simplified dimensionless form:
\begin{align}
 \label{eq:InterEg2}
W(\Gamma)=\int_{\Gamma}\gamma(\theta)ds-\sigma(x_c^r-x_c^l), 
 \end{align}
where $\sigma=\frac{\gamma_{VS}-\gamma_{FS}}{\gamma_0}$ is the dimensionless material constant with $\gamma_0$ denoting a dimensionless unit of surface energy density.
 
Define a time independent variable $\rho\in \mathbb I=[0, 1]$, and then we parameterize the curve $\Gamma(t)$ as  
 \begin{align}
 \label{eq:Interfacep}
 \Gamma(t):=\vec X(\rho,~t)=\big(x(\rho,t),~y(\rho,~t)\big)^T: 
 \mathbb I\times [0, T]\rightarrow \mathbb{R}^2.
 \end{align}
The arclength parameter $s$ can be computed by $s(\rho, t)=\int_0^\rho|\partial_q\vec{X}|dq$, and there holds  $\partial_\rho s=|\partial_\rho\vec{X}|$. 
By calculating variations with respect to $\Gamma$ and the two contact points $x_c^l$ and $x_c^r$, the following sharp-interface model is constructed for the solid-state dewetting of thin films with isotropic/weakly anisotropic surface energies \cite{Bao17}

 \begin{subequations}
 \label{eqn:model}
 \begin{align}
 \label{eq:model1}
 &\partial_t\vec X = \partial_{ss}\mu\vec{n},\quad 0<s<L(t),\quad t>0,\\
&\mu=\left[\gamma(\theta)+\gamma''(\theta)\right]\kappa,\qquad \kappa = -(\partial_{ss}\vec X)\cdot\vec n,
 \label{eq:model2}
 \end{align}
\end{subequations}
where $\mu$ is the chemical potential, $L(t)=\int_{\Gamma(t)}ds$ is the perimeter of $\Gamma(t)$, and $\kappa:=\kappa(s, t)$ is the curvature of the interface curve. 
The initial 
curve is defined by  
\begin{align}
\label{eq:modelinitial}	
\vec{X}(s, 0):=\vec{X}_0(s)=(x_0(s), y_0(s))^T,\quad 0\leq s\leq L_0:=L(0),
\end{align}
satisfying $y_0(0)=y_0(L_0)=0$ and $x_0(0)<x_0(L_0)$, and the boundary conditions are given by
\begin{itemize}
    \item [(i)] contact line condition
 \begin{align}
 \label{eq:bd1_2d}
 y(0, t) = y(L, t)=0,\qquad t\geq 0;
 \end{align}
 \item [(ii)] relaxed contact angle condition
 \begin{align}
 \label{eq:bd2_2d}
 \frac{dx_c^l}{dt}=\eta f(\theta_d^l; \sigma),\quad 
 \frac{dx_c^r}{dt}=-\eta f(\theta_d^r; \sigma),\qquad t\geq 0;
 \end{align}
 \item [(iii)] zero-mass flux condition
\begin{align}
\label{eq:bd3_2d}
\partial_s\mu(0, t)=0,\quad \partial_s\mu(L, t)=0,\quad t\geq 0,
\end{align}
\end{itemize}
where $\eta\in (0, \infty)$ represents the contact line mobility, 
$\theta_d^l$ and $\theta_d^r$ denote the dynamic contact angles at the left and right contact points, respectively, 
and 
\begin{align}\label{eq:bd_f}
    f(\theta; \sigma) 
    =\gamma(\theta)\cos\theta-\gamma'(\theta)\sin\theta-\sigma,\qquad \theta\in [-\pi, \pi],
\end{align}
with $f(\theta; \sigma)=0$ being the anisotropic Young equation \cite{Wang15}.
The contact line condition ensures that the contact points always move along the substrate, the relaxed contact angle condition allows for the relaxation of the contact angle, and the zero-mass flux condition ensures the conservation of the total area or mass of the thin film.

For the isotropic surface energy, there holds 
$\gamma(\theta)\equiv 1$ , and for the weakly anisotropic surface energy, it satisfies 
$\widetilde\gamma(\theta):=\gamma(\theta)+\gamma''(\theta)>0$ for all $\theta \in [-\pi, \pi]$. If  $\widetilde\gamma(\theta)<0$ for some orientations $\theta\in [-\pi, \pi]$, it is regarded to be strongly anisotropic. In such cases, there exist sharp corners in the equilibrium shape, and the sharp-interface governing model \eqref{eqn:model} turns ill-posedness. One of the efficient methods is to add 
some regularization terms, to make the sharp-interface model well-posedness. 
From \cite{Bao17,Jiang19a}, a Willmore energy regularization term was added in the original interfacial energy $W(\Gamma)$, given by 
\begin{align}
\label{eq:InterEg3}
W_{reg}^{\epsilon}(\Gamma):=W(\Gamma)+\epsilon^2W_{wm}(\Gamma)=
\int_\Gamma\left[\gamma(\theta)+\frac{\varepsilon^2}{2}\kappa^2\right]ds-\sigma(x_c^r-x_c^l),
\end{align}
where $0<\varepsilon\leq 1$ is a small regularization parameter and $W_{wm}(\Gamma)=\frac12\int_\Gamma\kappa^2ds$ is the Willmore energy.


Through obtaining the first variation of the surface energy function \eqref{eq:InterEg3} with respect to $\Gamma$ and the two contact points $x_c^l$ and $x_c^r$,  
a dimensionless regularized sharp-interface model for solid-state dewetting  of thin films with strongly anisotropic surface energy is derived: 
 \begin{subequations}
 \label{eqn:modelstr}
 \begin{align}
 \label{eq:model1str}
 &\partial_t\vec X = \partial_{ss}\mu\vec{n},\quad 0<s<L(t),\quad t>0,\\
&\mu=\left[\gamma(\theta)+\gamma''(\theta)\right]\kappa-\varepsilon^2\bigg(\partial_{ss}\kappa+\frac{1}{2}\kappa^3\bigg),\qquad \kappa = -(\partial_{ss}\vec X)\cdot\vec n.\label{eq:model2str}
 \end{align}
\end{subequations}
The initial curve is determined by \eqref{eq:modelinitial}, along with the boundary conditions provided below:
\begin{itemize}
    \item [(i)] contact line condition
 \begin{align}
 \label{eq:bd1_str}
 y(0, t) = 0,\quad y(L, t)=0,\quad t\geq 0;
 \end{align}
 \item [(ii)] relaxed contact angle condition
 \begin{align}
 \label{eq:bd2_str}
 \frac{dx_c^l}{dt}=\eta f^\epsilon(\theta_d^l; \sigma),\quad 
 \frac{dx_c^r}{dt}=-\eta f^\epsilon(\theta_d^r; \sigma),\qquad t\geq 0;
 \end{align}
 \item [(iii)]  zero-mass flux condition
\begin{align}
\label{eq:bd3_str}
\partial_s\mu(0, t)=0,\quad \partial_s\mu(L, t)=0,\quad t\geq 0;
\end{align}
\item [(iv)] zero-curvature condition
\begin{align}
\label{eq:bd4_str}
\kappa(0, t)=0,\quad \kappa(L, t)=0,\quad t\geq 0,
\end{align}
\end{itemize}
where the function $f^\epsilon(\theta; \sigma)$ is defined by 
\begin{align}\label{eq:bd_f_ep}
    f^\epsilon(\theta; \sigma) 
    =f^\epsilon(\theta; \sigma)-\epsilon^2\partial_s\kappa\sin\theta=\gamma(\theta)\cos\theta-\gamma'(\theta)\sin\theta-\sigma-\epsilon^2\partial_s\kappa\sin\theta,\qquad \theta\in [-\pi, \pi],
\end{align}
satisfying $\lim\limits_{\epsilon\rightarrow 0^+}f^\epsilon(\theta; \sigma)=f(\theta; \sigma)$.
Although the authors in \cite{Bao17,Jiang19a} have introduced the above regularized system, constructing an energy-stable parametric finite element approximation for the regularized sharp-interface model  remains an open problem. Additionally, apart from addressing the well-posedness of the model, other advantages of the regularization system have not yet been discovered, making it unclear in what other aspects the regularized model is superior to the original model. These two issues constitute the two important components of this paper. 
Most recently, Bao and Li introduced two new geometric identities in \cite{bao2024energy}, leading to the development of an energy-stable parametric finite element approximation for the planar Willmore flow. These geometric identities happen to be particularly useful for constructing the energy-stable algorithm for the regularized sharp-interface model \eqref{eqn:modelstr}. 
In this work, by introducing two types of novel surface energy matrices with related to the variable $\theta$, together with the two new geometric identities, 
we build the equivalent system of the regularized sharp-interface model \eqref{eqn:modelstr}. 
These matrices are slightly different from those in \cite{li2021energy,bao2023symmetrized, bao2023symmetrized1}, because all the functions in these matrices are related to the variable \(\theta\), 
and also different from one in \cite{li2021energy} as we add a stabilized term in each matrix in order 
to derive the energy stability of the parametric finite element method (PFEM). 
Based on the equivalent system and its weak formulation, we develop an energy-stable parametric finite element approximation with a theoretically proven energy stability law, addressing a significant gap in the current research. Furthermore, by adopting the techniques from \cite{bao2021structure}, we also construct a structure-preserving parametric finite element scheme that preserves both the area conservation and energy stability. 

However, the evolution meshes generated by the energy-stable scheme proposed in \cite{bao2024energy} fail to maintain uniform distribution, which may result in degenerate meshes over long-time simulations. This naturally raises concerns about whether applying this technique to strongly anisotropic solid-state dewetting models might lead to even worse mesh degeneration. But fortunately, through extensive numerical experiments, we demonstrate that this concern is unnecessary. The mesh quality produced by the numerical scheme with Willmore regularization is significantly better than that without it \cite{li2023symmetrized}, making it more suitable for long-term simulations. This constitutes the second main innovation of this work.
Mesh quality is a crucial indicator of the effectiveness of numerical schemes. Currently, the numerical methods that can produce high-quality meshes include the BGN methods developed in 
\cite{Barrett07,Barrett08JCP,Barrett2011,barrett2008parametric,Barrett20}, the DeTurck techniques
introduced in \cite{m2017approximations}, and the artificial tangential velocity defined in \cite{hu2022evolving} and Duan-Li's method \cite{duan2024new}. 

The rest of the paper is organized as follows. In Section \ref{sec2}, we present a unified surface energy matrix and derive a novel geometric system. In Section \ref{sec3}, we derive a new variational formula, and demonstrate its area conservation and energy dissipation laws. In Section \ref{sec4} and \ref{sec5}, We introduced the parametric finite element approximation and the solution process. In Section \ref{sec6}, we present extensive numerical tests to demonstrate the accuracy, structure-preserving properties, and efficiency of the energy-stable algorithms. Finally, we draw some conclusions in Section \ref{sec7}.


\section{A new geometric system}\label{sec2}
Define $\mat{B}_q(\theta)$ as 
\begin{equation}\label{Matrix:Bqx}
\mat{B}_{q}(\theta)=
\begin{pmatrix}
\gamma(\theta)&-\gamma'(\theta) \\
\gamma'(\theta)&\gamma(\theta)
\end{pmatrix}
\begin{pmatrix}
\cos2\theta&\sin2\theta \\
\sin2\theta&-\cos2\theta
\end{pmatrix}^{1-q}+
\mathscr S(\theta)\left[\frac{1}{2}\mat{I}-\frac{1}{2}\begin{pmatrix}
\cos2\theta&\sin2\theta \\
\sin2\theta&-\cos2\theta
\end{pmatrix}\right], \qquad q=0, 1,
\end{equation}
where $\mat I$ is a $2\times 2$ identity matrix and $\mathscr S(\theta)$ is a stability function.  For different values of $q$, the matrix includes the following two cases:
    \begin{itemize}
        \item [\textbf{1)}] For $q=0$,  $\mat{B}_0(\theta)$ is a symmetric matrix. Moreover, if $\gamma(\theta)=\gamma(\pi+\theta)$, we can prove that the matrix $\mat{B}_0(\theta)$ is also 
   positive definite. 
        \item [\textbf{2)}] For $q=1$, $\mat{B}_1(\theta)$ is not symmetric, and can be split into the symmetric positive matrix $\overline{\mat{B}}_1(\theta)$ and anti-symmetric matrix $\underline{\mat{B}}_1(\theta)$:
        \begin{align*}
\overline{\mat{B}}_1(\theta)    =\begin{pmatrix}
\gamma(\theta)&0 \\
0&\gamma(\theta)
\end{pmatrix}+
\mathscr S(\theta)\begin{pmatrix}
\frac{1-\cos2\theta}{2}&-\frac{1}{2}\sin2\theta \\
-\frac{1}{2}\sin2\theta&\frac{1+\cos2\theta}{2}
\end{pmatrix},
\qquad 
\underline{\mat{B}}_1(\theta)    =\begin{pmatrix}
0&-\gamma'(\theta) \\
\gamma'(\theta)&0
\end{pmatrix}.
        \end{align*}
    \end{itemize}
From \cite{bao2023symmetrized, bao2023symmetrized1,li2024structure}, we can prove 
\begin{align}\label{eqn:geoa}
\left[\gamma(\theta)+\gamma''(\theta)\right]\kappa\vec n=   
-\partial_s\left[ \mat{B}_{q}(\theta)\partial_s\vec{X}\right].
\end{align}
In addition, from \cite{bao2024energy}, we have the geometric relations:
\begin{align}
   & \left[\partial_{ss}\kappa+\frac12\kappa^3\right]\vec n=\partial_s\left[\partial_s\kappa\vec n-\frac12\kappa^2\partial_s\vec X\right],\label{eqn:geob}\\
   &\partial_t\kappa = -\partial_s\left[\vec n \cdot \partial_s \partial_t\vec X\right]- \left[\partial_s \vec X\cdot\partial_s\partial_t\vec X\right]\kappa.\label{eqn:geoc}
\end{align}
Therefore, the regularized sharp-interface model  \eqref{eqn:modelstr}  is equivalent to 
the following conservative form:
 \begin{subequations}
 \label{eqn:model2str}
 \begin{align}
  \label{eq:model4str}
 &\vec{n}\cdot\partial_t\vec X = \partial_{ss}\mu,\\
 &\mu\vec{n}=-\partial_s\left[ \mat{B}_{q}(\theta)\partial_s\vec{X}+\epsilon^2\left(\partial_s\kappa\vec n - \frac{1}{2}\kappa ^2\partial_s\vec X\right)\right], \label{eq:model5str}\\
& \partial_t\kappa = -\partial_s\left[\vec n \cdot \partial_s \partial_t\vec X\right]- \left[\partial_s \vec X\cdot\partial_s\partial_t\vec X\right]\kappa,
 \label{eq:model6str}
 \end{align}
\end{subequations}
with the initial condition \eqref{eq:modelinitial} and boundary conditions \eqref{eq:bd1_str}-\eqref{eq:bd4_str}. 
For this new equivalent system, we can easily prove its area conservation and energy dissipation laws.

\section{Variational formulation}\label{sec3}

 For the open evolution curve $\Gamma(t)$, we define the functional space 
 \begin{equation}
 \label{eq:L2space}
L^2(\mathbb{I}):=\left\{u: \mathbb I\rightarrow \mathbb R,~\int_{\Gamma(t)}|u(s)|^2ds=\int_{\mathbb I}|u(s(\rho, t))|^2\partial_\rho sd\rho<+\infty\right\},
\end{equation}
equipped with the $L^2$ inner product 
\begin{equation}
\label{eq:L2inner}
(u, v)_{\Gamma(t)}:=\int_{\Gamma(t)}u(s)v(s)ds=\int_{\mathbb I}u(s(\rho, t))v(s(\rho, t))\partial_\rho sd\rho,\quad \forall u, v\in L^2(\mathbb I).
\end{equation}
The above inner product can be extended to $[L^2(\mathbb I)]^2$ by replacing the scalar 
product $uv$ by the vector inner product $\vec{u}\cdot \vec{v}$. Define the Sobolev spaces 
\begin{subequations}
 \begin{align*}
 	&H^1(\mathbb I):=\Big\{u:\mathbb I\rightarrow \mathbb R, u\in L^2(\mathbb I)~ \text{and}~\partial_\rho u\in L^2(\mathbb I)\Big\},\\
 	&H_0^1(\mathbb I):=\Big\{u:\mathbb I\rightarrow \mathbb R, u\in H^1(\mathbb I)~ \text{and}~u(0)=u(1)=0\Big\},
 \end{align*}
\end{subequations}
and denote the space $\mathbb X:=H^1(\mathbb I)\times H_0^1(\mathbb I)$. 
Multiplying a test function $\phi\in H^1(\mathbb I)$ to \eqref{eq:model4str} and integrating over $\Gamma (t)$, employing integration by parts and combining boundary conditions \eqref{eq:bd3_str} and $\partial_\rho s = \left | \partial_\rho \vec X \right |, \text{d}s = \partial_\rho s \text{d}\rho = \left | \partial_\rho \vec X \right |\text{d}\rho$, we can obtain 
\begin{align}\label{bianfeng_1}
    \left(\vec n\cdot\partial_t\vec X, \phi\right)_{\Gamma(t)}
    &=\left(\partial_{ss}\mu,  
\phi\right)_{\Gamma(t)}   \nonumber \\ 
 &=-\left(\partial_s\mu, \partial_s\phi\right)_{\Gamma(t)} + \left(\phi \partial_s\mu  
\right)\bigg|_{\rho=0}^{\rho=1} \nonumber \\
&=-\left(\partial_s\mu, \partial_s \phi\right)_{\Gamma(t)}. 
\end{align}
Multiplying a test function $\vec\omega = \left(\omega_1, \omega_2\right)^\top\in\mathbb X$ to \eqref{eq:model5str} and integrating over $\Gamma(t)$, integrating by parts and noting the boundary conditions \eqref{eq:bd2_str}, \eqref{eq:bd4_str}
and the function $f^\epsilon(\theta; \sigma)$, we derive 
\begin{align}\label{bianfeng_2}
\bigg(\mu\vec n, \vec\omega\bigg)_{\Gamma(t)} &= -\left(\partial_s\left[ \mat{B}_{q}(\theta)\partial_s\vec{X}+\epsilon^2\left(\partial_s\kappa\vec n - \frac{1}{2}\kappa ^2\partial_s\vec X\right)\right], \vec\omega\right)_{\Gamma(t)} \nonumber \\
&= \left(\left[ \mat{B}_{q}(\theta)\partial_s\vec{X}+\epsilon^2\left(\partial_s\kappa\vec n - \frac{1}{2}\kappa ^2\partial_s\vec X\right)\right], \partial_s\vec\omega\right)_{\Gamma(t)} - \left(\left[ \mat{B}_{q}(\theta)\partial_s\vec{X}+\epsilon^2\left(\partial_s\kappa\vec n - \frac{1}{2}\kappa ^2\partial_s\vec X\right)\right]\cdot\vec\omega\right)\bigg|_{\rho=0} ^{\rho=1} \nonumber \\
&=\left( \mat{B}_{q}(\theta)\partial_s\vec{X}, \partial_s\vec\omega\right)_{\Gamma(t)}+ \epsilon^2\left(\left[\partial_s\kappa\vec n - \frac{1}{2}\kappa ^2\partial_s\vec X\right], \partial_s\vec\omega \right)_{\Gamma(t)} \nonumber \\
&~~~~+\frac{1}{\eta}\left[\frac{\text{d}x_c^l}{\text{d}t}\omega_1(0) + \frac{\text{d}x_c^r}{\text{d}t}\omega_1(1)\right] - \sigma\left[\omega_1(1)-\omega_1(0)\right].  
\end{align}
Then, multiplying a test function $\varphi\in H_0^1(\mathbb I)$ to \eqref{eq:model6str} and integrating by parts, we obtain 
\begin{equation}\label{bianfeng_3}
\bigg(\partial_t \kappa, \varphi\bigg)_{\Gamma(t)} = \left(\vec n\cdot\partial_s\partial_t\vec X, \partial_s\varphi\right)_{\Gamma(t)} - \left(\left[ \partial_s \vec X\cdot \partial_s \partial_t \vec X\right]\kappa, \varphi\right)_{\Gamma(t)}. 
\end{equation}
Combining \eqref{bianfeng_1}-\eqref{bianfeng_3}, we get a new variational formulation for the regularized solid-state dewetting \eqref{eqn:model2str} with the boundary conditions \eqref{eq:bd1_str}-\eqref{eq:bd4_str} as follows: Given an an initial curve $\vec{X}(s, 0):=\vec{X}_0(s)=(x_0(s), y_0(s))^T$, find the evolution curves $\vec{X}(\cdot, t)=(x(\cdot, t), y(\cdot, t))^T\in \mathbb X$, chemical potential $\mu(\cdot, t)\in H^1(\bI)$ and curvature $\kappa(\cdot, t)\in H_0^1(\bI)$, such that 
 \begin{subequations}
 \label{eqn:vf}
\begin{align}
    &\left(\vec n\cdot\partial_t\vec X, \phi\right)_{\Gamma(t)}
    +\left(\partial_s\mu, \partial_s \phi\right)_{\Gamma(t)}=0,\qquad \forall \phi\in H^1(\bI),\label{eqn:vfa}\\
    &\bigg(\mu\vec n, \vec\omega\bigg)_{\Gamma(t)}
    -\left( \mat{B}_{q}(\theta)\partial_s\vec{X}, \partial_s\vec\omega\right)_{\Gamma(t)}- \epsilon^2\left(\left[\partial_s\kappa\vec n - \frac{1}{2}\kappa ^2\partial_s\vec X\right], \partial_s\vec\omega \right)_{\Gamma(t)} \nn\\
    &~~~~~~
    -\frac{1}{\eta}\left[\frac{\text{d}x_c^l}{\text{d}t}\omega_1(0) + \frac{\text{d}x_c^r}{\text{d}t}\omega_1(1)\right] + \sigma\left[\omega_1(1)-\omega_1(0)\right]=0,
    \qquad \forall \vec\omega\in \mathbb X, \label{eqn:vfb}\\
    &\bigg(\partial_t \kappa, \varphi\bigg)_{\Gamma(t)} - \left(\vec n\cdot\partial_s\partial_t\vec X, \partial_s\varphi\right)_{\Gamma(t)} + \left(\left[ \partial_s \vec X\cdot \partial_s \partial_t \vec X\right]\kappa, \varphi\right)_{\Gamma(t)} = 0, \qquad \forall \varphi\in H_0^1(\bI). 
    \label{eqn:vfc}
\end{align}
\end{subequations}

Denote $A(t)$ as the total area enclosed between the curve $\Gamma(t)$ and the substrate, and let $W_c (t)$ denote the total interfacial energy, which are defined as follows: 
\begin{equation}
A(t):=\int_{\Gamma(t)} y(s, t)\partial_s x(s, t)\text{d}s, \quad
W_c (t):=\int_{\Gamma(t)} \gamma(\theta)\text{d}s + \frac{\varepsilon ^2}{2}\int_{\Gamma(t)} \kappa ^2\text{d}s - \sigma\left[x_c^r(t) - x_c^l(t)\right], \quad t\ge 0. 
\end{equation}
Then we immediately derive the area conservation and energy dissipation of the new variational formulation \eqref{eqn:vf}. 
\begin{lem}
(Area conservation \& Energy dissipation). Let $\left(\vec X(\cdot, t), \mu(\cdot, t),\kappa(\cdot, t)\right)\in\mathbb X \times H^1(\bI) \times H_0^1(\bI) $ be a solution of the variational formulation \eqref{eqn:vf}, then the area $A(t)$ is conservative during the evolution, i.e., 
\begin{align}\label{mass}
&A(t)\equiv A(0) = \int_{\Gamma(0)} y(s, 0)\partial_s x(s, 0)\text{d}s, \quad t\ge 0, 
\end{align}
and the energy $W_c(t)$ is dissipative during the evolution, i.e.,
\begin{align}\label{energy}
&W_c(t)\le W_c(t_1)\le W_c(0)=\int_{\Gamma(0)} \gamma(\theta)\text{d}s + \frac{\varepsilon ^2}{2}\int_{\Gamma(0)} \kappa ^2 \text{d}s - \sigma\left[x_c^r(0) - x_c^l(0)\right], \quad t\ge t_1\ge 0.
\end{align}
\end{lem}
\begin{proof}
Differentiaing $A(t)$ with respect to $t$, employing integration by parts and combining the boundary conditions, we have
\begin{align}
\frac{\text{d}}{\text{d}t}A(t) 
&= \frac{\text{d}}{\text{d}t}\int_{0}^{L(t)} y(s, t)\partial_s x(s, t)\text{d}s = \int_{0}^{1}\left(\partial_t y\partial_\rho x + y\partial_t\partial_\rho x \right)\text{d}\rho \nonumber \\
&= \int_{0}^{1}\left(\partial_t y\partial_\rho x - \partial_\rho y \partial_t x \right)\text{d}\rho + \left(y\partial_t x\right)\bigg|_{\rho = 0}^{\rho = 1} \nonumber \\
&= \int_{\Gamma(t)}\partial_t \vec X\cdot\vec n\text{d}s = \left(\partial_t\vec X, \vec n\right)_{\Gamma(t)}, \quad t\ge 0.
\end{align}
Selecting $\phi = 1$ in \eqref{eqn:vfa} , we can get
\begin{equation}
\frac{\text{d}}{\text{d}t}A(t) = \left(\partial_t \vec X, \vec n\right)_{\Gamma(t)} = -\left(\partial_s \mu, \partial_s 1\right)_{\Gamma(t)} = 0, \quad t\ge 0, 
\end{equation}
which implies the area conservation given in \eqref{mass}. 

Additionally, differentiating the energy $W_c(t)$ with respect to $t$, using integration by parts and noting the relations 
\begin{align}
&\partial_t \theta = -\frac{\left(\partial_\rho\vec X\right)^\bot\cdot\partial_t\partial_\rho\vec X}{\left | \partial_\rho\vec X \right | ^2}, \quad \partial_t\partial_\rho s = \frac{\partial_\rho\vec X\cdot\partial_t\partial_\rho\vec X}{\left | \partial_\rho\vec X \right |}, \quad \partial_\rho\theta = -\kappa\partial_\rho s, 
\nonumber \\
&\partial_t \kappa = -\partial_s\left(\vec n\cdot\partial_s\partial_t\vec X\right) - \left(\partial_s\vec X\cdot\partial_s\partial_t\vec X\right)\kappa,  \nonumber 
\end{align}
by virtue of \eqref{eqn:geoa} and \eqref{eqn:geob}, we can get
\begin{align}
\frac{\text{d}}{\text{d}t}W_c(t)
&= \frac{\text{d}}{\text{d}t}\int_{\Gamma(t)} \gamma(\theta)\text{d}s + \frac{\varepsilon ^2}{2}\frac{\text{d}}{\text{d}t}\int_{\Gamma(t)} \kappa ^2\text{d}s - \sigma\left[\frac{x_c^r(t)}{\text{d}t} - \frac{x_c^l(t)}{\text{d}t}\right] \nonumber \\
&= \int_{0}^{1}\gamma '(\theta)\partial_t \theta\partial_\rho s\text{d}\rho + \int_{0}^{1}\gamma(\theta)\partial_t\partial_\rho s\text{d}\rho + \varepsilon ^2\left(\int_{0}^{1}\kappa\partial_t\kappa\partial_\rho s\text{d}\rho + \int_{0}^{1}\frac{\kappa ^2}{2}\partial_t \partial_\rho s\text{d}\rho\right) - \sigma\left[\frac{x_c^r(t)}{\text{d}t} - \frac{x_c^l(t)}{\text{d}t}\right] \nonumber \\
&= \int_{0}^{1}\gamma'(\theta)\partial_t\partial_\rho\vec X\cdot\vec n\text{d}\rho + \int_{0}^{1}\gamma(\theta)\partial_t\partial_\rho\vec X \cdot\vec\tau\text{d}\rho  \nonumber \\
&~~~~+ \varepsilon ^2 \left[\int_{\Gamma(t)}\kappa\left(-\partial_s\left(\vec n\cdot\partial_s\partial_t \vec X\right) - \left(\partial_s\vec X \cdot\partial_s\partial_t\vec X\right)\kappa \right)\text{d}s + \int_{\Gamma(t)}\frac{\kappa ^2}{2}\partial_s\partial_t \vec X\cdot\partial_s \vec X \text{d}s\right] - 
\sigma\left[\frac{x_c^r(t)}{\text{d}t} - \frac{x_c^l(t)}{\text{d}t}\right] \nonumber \\
&= -\int_{0}^{1}\partial_t \vec X\cdot\left[\gamma''(\theta)\partial_\rho\theta\vec n + \gamma'(\theta)\partial_\rho\vec n + \gamma'(\theta)\partial_\rho\theta\vec\tau + \gamma(\theta)\partial_\rho\vec\tau\right]\text{d}\rho + \left[\partial_t\vec X\cdot\left(\gamma'(\theta)\vec n + \gamma(\theta)\vec \tau\right)\right]\bigg|_{\rho=0}^{\rho=1} \nonumber \\
&~~~~ + \varepsilon ^2\left(\int_{\Gamma(t)}\left(\partial_s\kappa \vec n - \frac{\kappa ^2}{2}\partial_s \vec X\right)\cdot\partial_s\partial_t \vec X\text{d}s\right) + \varepsilon ^2\kappa\left(\vec n\cdot\partial_s\partial_t\vec X \right)\bigg|_{s=0}^{s=L(t)} - 
\sigma\left[\frac{x_c^r(t)}{\text{d}t} - \frac{x_c^l(t)}{\text{d}t}\right] \nonumber \\
&= \int_{\Gamma(t)}\left(\kappa\left(\gamma(\theta) + \gamma''(\theta)\right) - \varepsilon ^2\left(\partial_{ss}\kappa + \frac{1}{2}\kappa ^3\right)\right)\partial_t \vec X\cdot\vec n\text{d}s - \frac{1}{\eta}\left[\left(\frac{x_c^r(t)}{\text{d}t}\right)^2 + \left(\frac{x_c^l(t)}{\text{d}t}\right)^2\right] \nonumber \\
&=-\int_{\Gamma(t)} \partial_s\left[ \mat{B}_{q}(\theta)\partial_s\vec{X}+\epsilon^2\left(\partial_s\kappa\vec n - \frac{1}{2}\kappa ^2\partial_s\vec X\right)\right]\cdot\partial_t \vec Xds - \frac{1}{\eta}\left[\left(\frac{x_c^r(t)}{\text{d}t}\right)^2 + \left(\frac{x_c^l(t)}{\text{d}t}\right)^2\right]
\nonumber \\
&= \int_{\Gamma(t)}\mu\partial_t\vec X\cdot\vec n\text{d}s - \frac{1}{\eta}\left[\left(\frac{x_c^r(t)}{\text{d}t}\right)^2 + \left(\frac{x_c^l(t)}{\text{d}t}\right)^2\right].
\end{align}
Choosing $\vec\omega = \partial_t\vec X$ in \eqref{eqn:vfb} and $\phi = \mu$ in \eqref{eqn:vfa}, we obtain 
\begin{equation}
\frac{\text{d}}{\text{d}t}W_c(t) = \left(\mu\vec n, \partial_t\vec X\right)_{\Gamma(t)} - \frac{1}{\eta}\left[\left(\frac{x_c^r(t)}{\text{d}t}\right)^2 +  \left(\frac{x_c^l(t)}{\text{d}t}\right)^2\right] = -\left(\partial_s\mu, \partial_s\mu\right)_{\Gamma(t)} - \frac{1}{\eta}\left[
\left(\frac{x_c^r(t)}{\text{d}t}\right)^2 +  \left(\frac{x_c^l(t)}{\text{d}t}\right)^2\right]\le 0, \quad t\ge 0, 
\end{equation}
which implies that the energy is dissipative during the evolution. Therefore, we complete the proof. 
\end{proof}

\begin{rem}
Given that the area conservation and energy dissipation laws of the new variational formulation \eqref{eqn:vf} have been proven in above theorem, it is desirable to construct numerical approximations that also adhere to these two geometric flows at the discrete level.
The constructed variational formulation \eqref{eqn:vf} for the regularized solid-state dewetting model \eqref{eqn:model2str} is crucial for establishing the energy-stable numerical method that will be given in next section. 
\end{rem}
\section{Parametric finite element approximations}\label{sec4}
In this section, we construct two types of energy-stable PFEMs for the proposed new variational formulation \eqref{eqn:vf}. We uniformly divide $[0, T] = \cup _{j = 0}^{M-1}[t_m,  t_{m+1}]$ with time steps $\ttau_m = t_{m} - t_{m-1}$ for $m\geq 1$, and the domain $\mathbb I$ is divided into $\mathbb I = \cup _{j=1}^{J}[q_{j-1}, q_j]$ with the nodes $q_j = jh$ and space stepsizes $h = J^{-1}$. Subsequently, we define the finite element spaces 
\begin{equation}
\mathbb K ^h = \mathbb K ^h (\mathbb I) := \left\{u \in C(\mathbb I):u|_{\mathbb I_j}\in\mathbb P_1, \quad\forall j=1,2,\dots ,J\right\}\subseteq H^1(\mathbb I), \quad \mathbb X^h := \mathbb K^h \times \mathbb K_0^h, \quad \mathbb K_0^h := \mathbb K^h \cap H_0^1(\mathbb I), \nonumber
\end{equation}
which $\mathbb P_1$ denote the space of polynomials with degree at most $1$. 
Let $\Gamma ^m := \vec X ^m\in\mathbb X^h$, $\mu^m\in\mathbb K^h$ and $\kappa ^m\in\mathbb K_0^h$ be the approximations of curve $\Gamma(t_m):=\vec X(\cdot, t_m)\in\mathbb X$, chemical potential $\mu(\cdot, t_m)\in H^1(\bI)$ and curvature $\kappa(\cdot, t_m)\in H_0^1(\bI)$, respectively. The approximation solution $\Gamma ^m$ is made up of the line segments
\begin{equation*}
\vec h_j^m := \vec X^m(\rho_j) - \vec X^m(\rho_{j-1}), \quad j = 1, 2, \dots , J, 
\end{equation*}
with $| \vec h_j^m| $ denoting the length of $\vec h_j^m$. Thus the unit tangential vector $\vec\tau ^m$ and the outward unit normal vector $\vec n ^m$ on interval $\mathbb I_j$ can be computed by 
\begin{equation}
\vec\tau ^m|_{\mathbb I_j} = \frac{\vec h_j^m}{\left | \vec h_j^m \right |} := \vec\tau_j^m, \quad \vec n^m|_{\mathbb I_j} = -\frac{\left(\vec h_j^m\right)^\bot}{\left | \vec h_j^m \right |} := \vec n_j^m. 
\end{equation}
Furthermore, we can define the mass lumped inner product $(\cdot, \cdot)_{\Gamma ^m}^h$ for two functions $\vec u$ and $\vec v$ with possible jumps at the nodes $\left\{\rho_j\right\}_{j=0}^{J}$ as follows. 
\begin{equation}
(\vec u, \vec v)_{\Gamma ^m}^h := \frac{1}{2}h\sum_{j=1}^{J}\left[( \vec u\cdot \vec v)(\rho_j^-)+( \vec u\cdot \vec v)(\rho_{j-1}^+)\right],  
\end{equation}
where $\vec v(\rho_j^\pm ) = \lim_{\rho \to \rho_j^{\pm}} \vec v(\rho)$ for $0\le j\le J$. 

Using the backward Euler method in time, we establish an energy-stable parametric finite element approximation (ES-PFEM): Given the initial curve $\Gamma ^0 \in\mathbb X^h$, $\mu ^0 \in\mathbb K^h$ and $\kappa ^0\in\mathbb K_0^h$, find the evolution curve $\Gamma ^{m+1} := \vec X^{m+1}(\cdot) = (x^{m+1}(\cdot), y^{m+1}(\cdot))^\top\in\mathbb X^h$,  chemical potential $\mu^{m+1}(\cdot)\in\mathbb K^h$ and curvature $\kappa ^{m+1}\in\mathbb K_0^h$, such that
\begin{subequations}\label{ES-PFEM}
\begin{align}
&\left(\frac{\vec X^{m+1}-\vec X^m}{\ttau}\cdot\vec n^m, \phi ^h\right)_{\Gamma ^m}^h + \left(\partial_s\mu ^{m+1}, \partial_s \phi ^h\right)_{\Gamma ^m}^h = 0, \quad \forall\phi ^h\in\mathbb K^h,  \label{ES-PFEM1} \\
&\left(\mu ^{m+1}, \vec n^m\cdot\vec\omega ^h\right)_{\Gamma ^m}^h - \left( \mat{B}_{q}(\theta ^m)\partial_s\vec X^{m+1}, \partial_s\vec\omega ^h\right)_{\Gamma ^m}^h - \varepsilon ^2\left(\partial_s\kappa ^{m+1}\vec n^m - \frac{1}{2}\left(\kappa ^{m+1}\right)^2\partial_s\vec X^{m+1}, \partial_s\vec\omega ^h\right)_{\Gamma ^m}^h \nonumber \\
&~~~~-\frac{1}{\eta}\left(\frac{ x_l^{m+1} - x_l^m}{\ttau}\omega_1^h(0) + \frac{ x_r^{m+1} - x_r^m}{\ttau}\omega_1^h(1)\right) + \sigma\left(\omega_1^h(1) - \omega_1^h(0)\right) = 0, \quad \forall\vec\omega ^h\in\mathbb X^h, 
\label{ES-PFEM2}  \\
&\left(\frac{\kappa ^{m+1} - \kappa ^m}{\ttau}, \varphi ^h\right)_{\Gamma ^m}^h - \left(\vec n^m\cdot\partial_s\left(\frac{\vec X^{m+1} - \vec X^m}{\ttau}\right), \partial_s\varphi ^h\right)_{\Gamma ^m}^h \nonumber \\
&~~~~+ \left(\partial_s \vec X^{m+1}\cdot\partial_s\left(\frac{\vec X^{m+1} - \vec X^m}{\ttau}\right)\kappa ^{m+1}, \varphi ^h\right)_{\Gamma ^m}^h = 0, \quad \forall\varphi ^h\in\mathbb K_0^h.  \label{ES-PFEM3}
\end{align}
\end{subequations}

We can get the energy stability property for the ES-PFEM. 
\begin{thm}\label{thm:energy-stability}
(Energy stability) 
Assume $(\vec X^m, \mu ^m, \kappa ^m)\in(\mathbb X^h, \mathbb K^h, \mathbb K_0^h)$  be the solution of \eqref{ES-PFEM}, and 
define the discretized energy $W_c^m$ enclosed between the curve $\Gamma ^m$ and the substrate as 
\begin{equation}
W_c^m := \sum_{j=1}^{J}\left | \vec h_j^m \right |\gamma(\theta _j^m) + \frac{\varepsilon ^2}{4}\sum_{j=1}^{J}\left | \vec h_j^m \right |\left((\kappa ^m) ^2(\rho_{j-1}^+) + (\kappa ^m) ^2(\rho_j^-)\right) - \sigma\left(x_r^m - x_l^m\right),\qquad m=0,\ldots, M.
\end{equation}
Then there exists a minimal stabilizing function $\mathscr S_0(\theta)$, such that when $\mathscr S(\theta)\geq \mathscr S_0(\theta)$,  the discretized energy $W_c^m$ is dissipative in the sense that  
\begin{equation} \label{w_c^m}
W_c^{m+1}\le W_c^m\le\dots\le W_c^0 = \sum_{j=1}^{J}\left | \vec h_j^0 \right |\gamma(\theta _j^0) + \frac{\varepsilon ^2}{4}\sum_{j=1}^{J}\left | \vec h_j^0 \right |\left((\kappa ^0) ^2(\rho_{j-1}^+) + (\kappa ^0) ^2(\rho_j^-)\right) - \sigma\left(x_r^0 - x_l^0\right), \quad \forall m\ge 0. 
\end{equation}
\end{thm}
\begin{proof} We first focus on the case of $q = 0$. In this case, 
inspired by \cite{bao2023symmetrized},  $\mat{B}_{q}(\theta)$ is a symmetric positive definite matrix if $\mathscr S(\theta)\ge \mathscr S_0(\theta)$, 
with $\mathscr S(\theta)$ denoting a minimal stabilizing function and $\mathscr S_0(\theta)$ defined by 
\begin{equation}\label{eqn:q0_k0}
\mathscr S_0(\theta) = \text{inf}\bigg\{k(\theta)|\gamma(\theta)\mat{B}_{0}(\theta)(\cos\hat{\theta} , \sin(\hat{\theta}))^\top\cdot(\cos\hat{\theta} , \sin(\hat{\theta}))^\top\ge\gamma(\hat{\theta}), \quad\forall\hat{\theta}\in[-\pi, \pi]\bigg\}, \quad\theta\in [-\pi, \pi]. 
\end{equation}
Choosing $\phi ^h = \mu ^{m+1}\ttau$ in \eqref{ES-PFEM1}, $\vec\omega ^h = \vec X^{m+1} - \vec X^m$ in \eqref{ES-PFEM2}, and $\varphi ^h = \varepsilon ^2\kappa ^{m+1}\ttau$ in \eqref{ES-PFEM3}, we have 
\begin{subequations}
\begin{align}
&\left(\left(\vec X^{m+1}-\vec X^m\right)\cdot\vec n^m, \mu ^{m+1}\right)_{\Gamma ^m}^h + \ttau\left(\partial_s\mu ^{m+1}, \partial_s \mu ^{m+1}\right)_{\Gamma ^m}^h = 0, \label{ES-PFEM11} \\
&\left(\mu ^{m+1}, \vec n^m\cdot\left(\vec X^{m+1} - \vec X^m\right)\right)_{\Gamma ^m}^h - \left( \mat{B}_{0}(\theta ^m)\partial_s\vec X^{m+1}, \partial_s\left(\vec X^{m+1} - \vec X^m\right)\right)_{\Gamma ^m}^h - \varepsilon ^2\left(\partial_s\kappa ^{m+1}\vec n^m - \frac{1}{2}\left(\kappa ^{m+1}\right)^2\partial_s\vec X^{m+1}, \partial_s\left(\vec X^{m+1} - \vec X^m\right)\right)_{\Gamma ^m}^h \nonumber \\
&~~~~-\frac{1}{\eta}\left(\frac{ \left(x_l^{m+1} - x_l^m\right)^2}{\ttau} + \frac{\left( x_r^{m+1} - x_r^m \right)^2}{\ttau}\right) + \sigma\left(\left( x_r^{m+1} - x_r^m \right) - \left( x_l^{m+1} - x_l^m \right)\right) = 0,  \label{ES-PFEM22}\\
&\varepsilon ^2\left(\kappa ^{m+1} - \kappa ^m, \kappa ^{m+1}\right)_{\Gamma ^m}^h - \varepsilon ^2\left(\vec n^m\cdot\partial_s\left(\vec X^{m+1} - \vec X^m\right), \partial_s\kappa ^{m+1}\right)_{\Gamma ^m}^h 
+ \varepsilon ^2\left(\partial_s \vec X^{m+1}\cdot\partial_s\left(\vec X^{m+1} - \vec X^m\right)\kappa ^{m+1}, \kappa ^{m+1}\right)_{\Gamma ^m}^h = 0.  \label{ES-PFEM33}
\end{align}
\end{subequations}
Inserting \eqref{ES-PFEM22} and \eqref{ES-PFEM33} into \eqref{ES-PFEM11}, we can obtain 
\begin{align}\label{eqn:proof1}
&\varepsilon ^2\left(\kappa ^{m+1} - \kappa ^m, \kappa ^{m+1}\right)_{\Gamma ^m}^h + \frac{\varepsilon ^2}{2}\left(\left(\kappa ^{m+1}\right)^2\partial_s\vec X^{m+1}, \partial_s\left(\vec X^{m+1} - \vec X^m\right)\right)_{\Gamma ^m}^h + \left( \mat{B}_{0}(\theta ^m)\partial_s\vec X^{m+1}, \partial_s\left(\vec X^{m+1} - \vec X^m\right)\right)_{\Gamma ^m}^h \nonumber \\
&~~~~ - \sigma\left(\left( x_r^{m+1} - x_r^m \right) - \left( x_l^{m+1} - x_l^m \right)\right) = -\ttau\left(\partial_s\mu ^{m+1}, \partial_s \mu ^{m+1}\right)_{\Gamma ^m}^h - \frac{1}{\eta}\left(\frac{ \left(x_l^{m+1} - x_l^m\right)^2}{\ttau} + \frac{\left( x_r^{m+1} - x_r^m \right)^2}{\ttau}\right).
\end{align}
Using the positive definiteness  properties of the matrix $\mat{B}_0(\theta)$, we have 
\begin{align}\label{eqn:proof2}
&\varepsilon ^2\left(\kappa ^{m+1} - \kappa ^m, \kappa ^{m+1}\right)_{\Gamma ^m}^h + \frac{\varepsilon ^2}{2}\left(\left(\kappa ^{m+1}\right)^2\partial_s\vec X^{m+1}, \partial_s\left(\vec X^{m+1} - \vec X^m\right)\right)_{\Gamma ^m}^h \nonumber \\
&~~~~+ \left( \mat{B}_{0}(\theta ^m)\partial_s\vec X^{m+1}, \partial_s\left(\vec X^{m+1} - \vec X^m\right)\right)_{\Gamma ^m}^h 
 - \sigma\left(\left( x_r^{m+1} - x_r^m \right) - \left( x_l^{m+1} - x_l^m \right)\right) \nonumber \\
 &\ge \varepsilon ^2\left(\kappa ^{m+1} - \kappa ^m, \kappa ^{m+1}\right)_{\Gamma ^m}^h + \frac{\varepsilon ^2}{4}\left(\left(\kappa ^{m+1}\right)^2\partial_s\vec X^{m+1}, \partial_s\vec X^{m+1} \right)_{\Gamma ^m}^h + \frac{\varepsilon ^2}{4}\left(\left(\kappa ^{m+1}\right)^2\partial_s\vec X^m, \partial_s\vec X^m \right)_{\Gamma ^m}^h  \nonumber \\
 &~~~~+\frac{1}{2}\left( \mat{B}_{0}(\theta ^m)\partial_s\vec X^{m+1}, \partial_s\vec X^{m+1} \right)_{\Gamma ^m}^h - \frac{1}{2}\left( \mat{B}_{0}(\theta ^m)\partial_s\vec X^m, \partial_s\vec X^m \right)_{\Gamma ^m}^h - \sigma\left(\left( x_r^{m+1} - x_r^m \right) - \left( x_l^{m+1} - x_l^m \right)\right) \nonumber \\ 
 &\ge \frac{\varepsilon ^2}{2}\left(\kappa ^{m+1} , \kappa ^{m+1}\right)_{\Gamma ^m}^h + \frac{\varepsilon ^2}{4}\left(\left(\kappa ^{m+1}\right)^2\partial_s\vec X^{m+1}, \partial_s\vec X^{m+1} \right)_{\Gamma ^m}^h + \frac{\varepsilon ^2}{4}\left(\left(\kappa ^{m+1}\right)^2\partial_s\vec X^m, \partial_s\vec X^m \right)_{\Gamma ^m}^h - \frac{\varepsilon ^2}{2}\left(\kappa ^m , \kappa ^m\right)_{\Gamma ^m}^h  \nonumber \\
 &~~~~+\frac{1}{2}\left( \mat{B}_{0}(\theta ^m)\partial_s\vec X^{m+1}, \partial_s\vec X^{m+1} \right)_{\Gamma ^m}^h - \frac{1}{2}\left( \mat{B}_{0}(\theta ^m)\partial_s\vec X^m, \partial_s\vec X^m \right)_{\Gamma ^m}^h - \sigma\left(\left( x_r^{m+1} - x_r^m \right) - \left( x_l^{m+1} - x_l^m \right)\right) \nonumber \\ 
 &=\varepsilon ^2\sum_{j=1}^{J}\frac{\left | \vec h_j^m \right | + \frac{\left | \vec h_j^{m+1} \right |^2}{\left | \vec h_j^m \right |}}{8}\left(\left(\kappa ^{m+1}\right)^2 \left(\rho_{j-1} \right) + \left(\kappa ^{m+1}\right) ^2 \left(\rho_j\right)\right) - \frac{\varepsilon ^2}{4}\sum_{j=1}^{J}\left | \vec h_j^m \right | \left(\left(\kappa ^m\right)^2 \left(\rho_{j-1} \right) + \left(\kappa ^m\right) ^2 \left(\rho_j\right)\right) 
 \nonumber \\ 
 &~~~~+\sum_{j=1}^{J}\frac{\left(\vec h_j^{m+1} \right)^\top \mat {B}_0(\theta _j^m)\vec h_j^{m+1} + \gamma(\theta _j^m)\left | \vec h_j^m \right |^2}{2 \left | \vec h_j^m \right |} - \sum_{j=1}^{J}\left | \vec h_j^m \right | \gamma(\theta _j^m) - \sigma\left(\left( x_r^{m+1} - x_r^m \right) - \left( x_l^{m+1} - x_l^m \right)\right) \nonumber \\ 
 &\ge \frac{\varepsilon ^2}{4}\sum_{j=1}^{J}\left | \vec h_j^{m+1} \right | \left(\left(\kappa ^{m+1}\right)^2 \left(\rho_{j-1} \right) + \left(\kappa ^{m+1}\right) ^2 \left(\rho_j\right)\right) - \frac{\varepsilon ^2}{4}\sum_{j=1}^{J}\left | \vec h_j^m \right | \left(\left(\kappa ^m\right)^2 \left(\rho_{j-1} \right) + \left(\kappa ^m\right) ^2 \left(\rho_j\right)\right) \nonumber \\
 &~~~~+\sum_{j=1}^{J}\left | \vec h_j^{m+1} \right | \gamma(\theta _j^{m+1})-\sum_{j=1}^{J}\left | \vec h_j^m \right | \gamma(\theta _j^m) - \sigma\left(\left( x_r^{m+1} - x_r^m \right) - \left( x_l^{m+1} - x_l^m \right)\right) \nonumber \\ 
 &=W_c^{m+1}-W_c^m.
\end{align}
From \eqref{eqn:proof1} and \eqref{eqn:proof2}, we conclude 
\begin{align}
    W_c^{m+1}-W_c^m\leq 0. 
\end{align}
Hence, we complete the proof of the energy stability for the case of $q=0$. 

Subsequently, we consider the case of $q=1$. In this case, from \cite{bao2024structure}, existing a minimal stabilizing function $\mathscr S_0(\theta)$, if $\mathscr S(\theta)\ge \mathscr S_0(\theta)$, we can obtain the following inequality 
\begin{equation}\label{inequality}
\frac{1}{\left | \vec v \right |}\left(\mat{B}_1(\theta)\vec w\right)\cdot(\vec w-\vec v)\ge\left | \vec w \right |\gamma(\hat{\theta})-\left | \vec v \right |\gamma(\theta), 
\end{equation}
where $(-\sin \theta, \cos \theta) = \frac{\vec v}{\left | \vec v \right |}$, $(-\sin \hat{\theta}, \cos \hat{\theta}) = \frac{\vec w}{\left | \vec w \right |}$. Then taking $\phi ^h = \mu ^{m+1}\ttau$ in \eqref{ES-PFEM1}, $\vec\omega ^h = \vec X^{m+1} - \vec X^m$ in \eqref{ES-PFEM2}, and $\varphi ^h = \varepsilon ^2\kappa ^{m+1}\ttau$ in \eqref{ES-PFEM3}, we have 
\begin{align}\label{eqn:proof3}
&\varepsilon ^2\left(\kappa ^{m+1} - \kappa ^m, \kappa ^{m+1}\right)_{\Gamma ^m}^h + \frac{\varepsilon ^2}{2}\left(\left(\kappa ^{m+1}\right)^2\partial_s\vec X^{m+1}, \partial_s\left(\vec X^{m+1} - \vec X^m\right)\right)_{\Gamma ^m}^h + \left( \mat{B}_{1}(\theta ^m)\partial_s\vec X^{m+1}, \partial_s\left(\vec X^{m+1} - \vec X^m\right)\right)_{\Gamma ^m}^h \nonumber \\
&~~~~ - \sigma\left(\left( x_r^{m+1} - x_r^m \right) - \left( x_l^{m+1} - x_l^m \right)\right) = -\ttau\left(\partial_s\mu ^{m+1}, \partial_s \mu ^{m+1}\right)_{\Gamma ^m}^h - \frac{1}{\eta}\left(\frac{ \left(x_l^{m+1} - x_l^m\right)^2}{\ttau} + \frac{\left( x_r^{m+1} - x_r^m \right)^2}{\ttau}\right).
\end{align}
Then, it follows from \eqref{inequality} that 
\begin{align}\label{eqn:proof4}
&\varepsilon ^2\left(\kappa ^{m+1} - \kappa ^m, \kappa ^{m+1}\right)_{\Gamma ^m}^h + \frac{\varepsilon ^2}{2}\left(\left(\kappa ^{m+1}\right)^2\partial_s\vec X^{m+1}, \partial_s\left(\vec X^{m+1} - \vec X^m\right)\right)_{\Gamma ^m}^h \nonumber \\
&~~~~+ \left( \mat{B}_{1}(\theta ^m)\partial_s\vec X^{m+1}, \partial_s\left(\vec X^{m+1} - \vec X^m\right)\right)_{\Gamma ^m}^h 
 - \sigma\left(\left( x_r^{m+1} - x_r^m \right) - \left( x_l^{m+1} - x_l^m \right)\right) \nonumber \\
 &\ge \varepsilon ^2\left(\kappa ^{m+1} - \kappa ^m, \kappa ^{m+1}\right)_{\Gamma ^m}^h + \frac{\varepsilon ^2}{4}\left(\left(\kappa ^{m+1}\right)^2\partial_s\vec X^{m+1}, \partial_s\vec X^{m+1} \right)_{\Gamma ^m}^h + \frac{\varepsilon ^2}{4}\left(\left(\kappa ^{m+1}\right)^2\partial_s\vec X^m, \partial_s\vec X^m \right)_{\Gamma ^m}^h  \nonumber \\
&~~~~ + \sum_{j=1}^{J}\left | \vec h_j^{m+1} \right | \gamma(\theta _j^{m+1})-\sum_{j=1}^{J}\left | \vec h_j^m \right | \gamma(\theta _j^m) - \sigma\left(\left( x_r^{m+1} - x_r^m \right) - \left( x_l^{m+1} - x_l^m \right)\right) \nonumber \\
&\ge \frac{\varepsilon ^2}{4}\sum_{j=1}^{J}\left | \vec h_j^{m+1} \right | \left(\left(\kappa ^{m+1}\right)^2 \left(\rho_{j-1} \right) + \left(\kappa ^{m+1}\right) ^2 \left(\rho_j\right)\right) - \frac{\varepsilon ^2}{4}\sum_{j=1}^{J}\left | \vec h_j^m \right | \left(\left(\kappa ^m\right)^2 \left(\rho_{j-1} \right) + \left(\kappa ^m\right) ^2 \left(\rho_j\right)\right) \nonumber \\
 &~~~~ + \sum_{j=1}^{J}\left | \vec h_j^{m+1} \right | \gamma(\theta _j^{m+1})-\sum_{j=1}^{J}\left | \vec h_j^m \right | \gamma(\theta _j^m) - \sigma\left(\left( x_r^{m+1} - x_r^m \right) - \left( x_l^{m+1} - x_l^m \right)\right) \nonumber \\ 
 &=W_c^{m+1}-W_c^m. 
\end{align}
From \eqref{eqn:proof3} and \eqref{eqn:proof4}, we have 
\begin{align}
    W_c^{m+1}-W_c^m\leq 0. 
\end{align}
Therefore, we conclude the energy stability for the case of $q=1$. 
\end{proof}

\begin{rem}
Define the minimal stabilizing function $\mathscr S_0(\theta)$ as 
\begin{equation}
\mathscr S_0(\theta):=inf\left \{ \mathscr S(\theta)|\bigg[\gamma(\theta)\mat{B}_0(\theta)\,(\cos\hat{\theta}, \sin\hat{\theta})^\top\bigg]\cdot(\cos\hat{\theta}, \sin\hat{\theta})^\top\geq \gamma(\hat{\theta})^2, \forall\hat{\theta}\in[-\pi, \pi] \right \}, \quad\theta\in[-\pi, \pi]. 
\end{equation}
If $\gamma(\theta)=\gamma(\pi+\theta)$ and $\mathscr S(\theta)\ge \mathscr S_0(\theta)$ for $\theta\in[-\pi, \pi]$, we can obtain the result \eqref{eqn:q0_k0} (see Ref. \cite{bao2023symmetrized} for details). 

In addition, 
we define $P_\alpha(\theta, \hat{\theta})$ and $ Q(\theta, \hat{\theta})$ as 
\begin{subequations}
\begin{align}
&P_\alpha(\theta, \hat{\theta}):=2\sqrt{(\gamma(\theta)+\alpha(-\sin\hat{\theta}\cos\theta+\cos\hat{\theta}\sin\theta)^2)\gamma(\theta)}, \quad\forall\theta, \hat{\theta}\in[-\pi, \pi], \quad\alpha\ge 0, \\
&Q(\theta, \hat{\theta}):=\gamma(\hat{\theta})+\gamma(\theta)(\sin\theta\sin\hat{\theta}+\cos\theta\cos\hat{\theta})+\gamma'(\theta)(-\sin\hat{\theta}\cos\theta+\cos\hat{\theta}\sin\theta), \quad\forall\theta, \hat{\theta}\in[-\pi, \pi]. 
\end{align}
\end{subequations}
Then, we define the minimal stabilizing function $\mathscr S_0(\theta)$ as 
\begin{equation}
\mathscr S_0(\theta):=inf\left \{ \alpha\ge 0:P_\alpha(\theta, \hat{\theta})-Q(\theta, \hat{\theta})\ge 0, \forall\hat{\theta}\in[-\pi, \pi] \right \}, \quad\theta\in[-\pi, \pi]. 
\end{equation}
The inequality \eqref{inequality} can be demonstrated, provided that $(-\sin\theta, \cos\theta)^\top=-\frac{\vec v}{\left | \vec v \right |}$ and $(-\sin\hat{\theta}, \cos\hat{\theta})^\top=-\frac{\vec w}{\left | \vec w \right |}$ are nonzero vectors, 
$3\gamma(\theta)>\gamma(\pi + \theta)$ and the stabilizing function holds $\mathscr S(\theta)\ge\mathscr S_0(\theta)$ for $\theta\in[-\pi, \pi]$ (see Ref. \cite{bao2024structure} for details). 
\end{rem}

Except the ES-PFEM, we further propose the following area-conserving parametric finite element approximation (AC-PFEM): Given the initial curve $\Gamma ^0 \in\mathbb X^h$, $\mu ^0 \in\mathbb K^h$ and $\kappa ^0\in\mathbb K_0^h$, find the evolution curve $\Gamma ^{m+1} := \vec X^{m+1}(\cdot) = (x^{m+1}(\cdot), y^{m+1}(\cdot))^\top\in\mathbb X^h$,  chemical potential $\mu^{m+1}(\cdot)\in\mathbb K^h$ and curvature $\kappa ^{m+1}\in\mathbb K_0^h$, such that
\begin{subequations}\label{AC-PFEM}
\begin{align}
&\left(\frac{\vec X^{m+1}-\vec X^m}{\ttau}\cdot\vec n^{m+\frac{1}{2}}, \phi ^h\right)_{\Gamma ^m}^h + \left(\partial_s\mu ^{m+1}, \partial_s \phi ^h\right)_{\Gamma ^m}^h = 0, \quad \forall\phi ^h\in\mathbb K^h,  \label{AC-PFEM1} \\
&\left(\mu ^{m+1}, \vec n^{m + \frac{1}{2}}\cdot\vec\omega ^h\right)_{\Gamma ^m}^h - \left( \mat{B}_{q}(\theta ^m)\partial_s\vec X^{m+1}, \partial_s\vec\omega ^h\right)_{\Gamma ^m}^h - \varepsilon ^2\left(\partial_s\kappa ^{m+1}\vec n^m - \frac{1}{2}\left(\kappa ^{m+1}\right)^2\partial_s\vec X^{m+1}, \partial_s\vec\omega ^h\right)_{\Gamma ^m}^h \nonumber \\
&~~~~-\frac{1}{\eta}\left(\frac{ x_l^{m+1} - x_l^m}{\ttau}\omega_1^h(0) + \frac{ x_r^{m+1} - x_r^m}{\ttau}\omega_1^h(1)\right) + \sigma\left(\omega_1^h(1) - \omega_1^h(0)\right) = 0, \quad \forall\vec\omega ^h\in\mathbb X^h, 
\label{AC-PFEM2}  \\
&\left(\frac{\kappa ^{m+1} - \kappa ^m}{\ttau}, \varphi ^h\right)_{\Gamma ^m}^h - \left(\vec n^m\cdot\partial_s\left(\frac{\vec X^{m+1} - \vec X^m}{\ttau}\right), \partial_s\varphi ^h\right)_{\Gamma ^m}^h \nonumber \\
&~~~~+ \left(\partial_s \vec X^{m+1}\cdot\partial_s\left(\frac{\vec X^{m+1} - \vec X^m}{\ttau}\right)\kappa ^{m+1}, \varphi ^h\right)_{\Gamma ^m}^h = 0, \quad \forall\varphi ^h\in\mathbb K_0^h,  \label{AC-PFEM3}
\end{align}
\end{subequations} 
where $\vec n^{m+\frac{1}{2}}$ is defined by
\begin{equation}
\vec n^{m + \frac{1}{2}} = -\frac{1}{2}\left(\partial_s \vec X^m + \partial_s \vec X^{m+1}\right)^\perp = -\frac{1}{2}\left | \partial_\rho \vec X^m \right |^{-1} \left(\partial_\rho \vec X^m + \partial_\rho \vec X^{m+1}\right)^\perp.
\end{equation}

For the AC-PFEM, we can obtain both the area conservation and energy stability properties, shown in the following theorem. 
\begin{thm}
(Area conservation \& Energy stability) 
Assume $(\vec X^m, \mu ^m, \kappa ^m)\in(\mathbb X^h, \mathbb K^h, \mathbb K_0^h)$ be the solution of the AC-PFEM, and 
define the discretized area $A^m$ enclosed between the curve $\Gamma ^m$ and the substrate as 
\begin{equation}
A^m := \frac{1}{2}\sum_{j=1}^{J}(x_j^m - x_{j-1}^m)(y_j^m + y_{j-1}^m),\qquad m=0,\ldots, M. 
\end{equation}
Then it holds that
\begin{align}
\label{eqn:area-conservation}
A^m \equiv A^0 = \frac{1}{2}\sum_{j=1}^{J}(x_j^0 - x_{j-1}^0)(y_j^0 + y_{j-1}^0), \qquad m=0,\ldots, M. 
\end{align}
Additionally, we can also obtain the energy stability property given in \eqref{w_c^m}. 

\end{thm}
\begin{proof}
The area conservation property can be demonstrated similarly to the method introduced in \cite{bao2021structure,li2023symmetrized}. Meanwhile, the energy stability property follows similarly to Theorem \ref{thm:energy-stability}. For the sake of brevity, detailed proofs are omitted here.
\end{proof}
\begin{rem}
  In this work, we have developed two numerical schemes, namely ES-PFEM and AC-PFEM, both of which can be demonstrated to be energy-stable. Additionally, the AC-PFEM scheme also achieves area conservation. These achievements fill a notable gap in this area of research. Furthermore, through extensive numerical experiments presented in the next section, we show that compared to the numerical methods for the model without Willmore energy, the schemes we have constructed exhibit excellent mesh quality. In the near future, we will further study the energy-stable parameter finite element approximation algorithms for axisymmetric regularized solid-state dewetting problems and three-dimensional regularized solid-state dewetting problems.
\end{rem}
\section{An iterative solver}\label{sec5}

Since the ES-PFEM and AC-PFEM are both fully-implicit, to get the solutions of ES-PFEM and AC-PFEM, we use the Newton-Raphson iteration. 

For the ES-PFEM, we initialize the approximation $(\vec X^{m+1, 0}, \mu^{m+1, 0}, \kappa^{m+1, 0})$ with the solution from the previous time step $(\vec X^m, \mu^m, \kappa^m)$. In each iteration $i$, using the current approximation $(\vec X^{m+1, i}, \mu^{m+1, i}, \kappa^{m+1, i})$, the next approximation $(\vec X^{m+1, i+1}, \mu^{m+1, i+1}, \kappa^{m+1, i+1})$ is computed by 
\begin{subequations}
\begin{align}
&\left(\frac{\vec X^{m+1, i+1} - \vec X^m}{\ttau}\cdot \vec n^m, \phi ^h \right)_{\Gamma^m}^h + \left(\partial_s\mu ^{m+1, i+1}, \partial_s\phi ^h\right)_{\Gamma^m}^h = 0, \quad \forall\phi^h\in \mathbb K^h, \\
&\left(\mu ^{m+1, i+1}, \vec n^m\cdot\vec\omega ^h\right)_{\Gamma^m}^h - \left(\mat B_q(\theta ^m)\partial_s\vec X^{m+1, i+1}, \partial_s\vec\omega ^h\right)_{\Gamma^m}^h - \varepsilon ^2\left(\partial_s\kappa ^{m+1, i+1}\vec n^m, \partial_s\vec\omega ^h\right)_{\Gamma^m}^h \nonumber \\
&~~~~ + \varepsilon ^2\left(\kappa ^{m+1, i+1}\kappa ^{m+1, i}\partial_s\vec X^{m+1, i}, \partial_s\vec\omega ^h\right)_{\Gamma ^m}^h + \varepsilon ^2\left(\frac{1}{2}\left(\kappa ^{m+1, i}\right)^2\partial_s\vec X^{m+1, i+1}, \partial_s\vec\omega ^h\right)_{\Gamma^m}^h - \varepsilon ^2\left(\left(\kappa ^{m+1, i}\right)^2\partial_s\vec X^{m+1, i}, \partial_s\vec\omega ^h\right)_{\Gamma^m}^h \nonumber \\
&~~~~-\frac{1}{\eta}\left(\frac{ x_l^{m+1, i+1} - x_l^m}{\ttau}\omega_1^h(0) + \frac{ x_r^{m+1, i+1} - x_r^m}{\ttau}\omega_1^h(1)\right) + \sigma\left(\omega_1^h(1) - \omega_1^h(0)\right) = 0, \quad \forall\vec\omega ^h\in\mathbb X^h, \\
&\left(\frac{\kappa ^{m+1, i+1} - \kappa ^m}{\ttau}, \varphi ^h\right)_{\Gamma ^m}^h - \left(\vec n^m\cdot\partial_s\left(\frac{\vec X^{m+1, i+1} - \vec X^m}{\ttau}\right), \partial_s\varphi ^h\right)_{\Gamma ^m}^h - \frac{2}{\ttau}\left(\left(\partial_s\vec X^{m+1, i}\right)^2\kappa ^{m+1, i},\varphi ^h\right)_{\Gamma^m}^h \nonumber \\
&~~~~ + \frac{2}{\ttau}\left(\partial_s\vec X^{m+1, i+1}\partial_s\vec X^{m+1, i}\kappa ^{m+1, i},\varphi ^h\right)_{\Gamma^m}^h + \frac{1}{\ttau}\left(\left(\partial_s\vec X^{m+1, i}\right)^2\kappa ^{m+1, i+1},\varphi ^h\right)_{\Gamma^m}^h + \frac{1}{\ttau}\left(\partial_s\vec X^{m+1, i}\partial_s\vec X^m\kappa ^{m+1, i}, \varphi ^h\right)_{\Gamma^m}^h \nonumber \\
&~~~~- \frac{1}{\ttau}\left(\partial_s\vec X^{m+1, i+1}\partial_s \vec X^m\kappa ^{m+1, i}, \varphi ^h\right)_{\Gamma^m}^h - \frac{1}{\ttau}\left(\partial_s\vec X^{m+1, i}\partial_s \vec X^m\kappa ^{m+1, i+1}, \varphi ^h\right)_{\Gamma ^m}^h = 0, \quad \forall\varphi ^h\in\mathbb K_0^h.
\end{align}
\end{subequations}

For the AC-PFEM, we initialize the approximation $(\vec X^{m+1, 0}, \mu^{m+1, 0}, \kappa^{m+1, 0})$ with the solution from the previous time step $(\vec X^m, \mu^m, \kappa^m)$. In each iteration $i$, by combining $\vec n^{m + \frac{1}{2}} = -\frac{1}{2}\left(\partial_s \vec X^m + \partial_s \vec X^{m+1}\right)^\perp$, using the current approximation $(\vec X^{m+1, i}, \mu^{m+1, i}, \kappa^{m+1, i})$, the next approximation $(\vec X^{m+1, i+1}, \mu^{m+1, i+1}, \kappa^{m+1, i+1})$ is computed by 
\begin{subequations}
\begin{align}
&-\frac{1}{2}\left(\frac{\vec X^{m+1, i+1} - \vec X^m}{\ttau}\cdot \partial_s\vec X^m, \phi ^h\right)_{\Gamma^m}^h + 
\frac{1}{2\ttau}\left(\vec X^m\partial_s\vec X^{m+1, i+1}, \phi ^h\right)_{\Gamma ^m}^h  
+ \frac{1}{2\ttau}\left(\vec X^{m+1, i}\partial_s\vec X^{m+1, i}, \phi ^h\right)_{\Gamma ^m}^h
\nonumber \\
&~~~~ - \frac{1}{2\ttau}\left(\vec X^{m+1, i+1}\partial_s\vec X^{m+1, i}, \phi ^h\right)_{\Gamma ^m}^h - \frac{1}{2\ttau}\left(\vec X^{m+1, i}\partial_s\vec X^{m+1, i+1}, \phi ^h\right)_{\Gamma ^m}^h + 
\left(\partial_s\mu ^{m+1, i+1}, \partial_s\phi ^h\right)_{\Gamma^m}^h 
= 0, \quad \forall\phi^h\in \mathbb K^h, \\
&-\frac{1}{2}\left(\mu ^{m+1, i+1}, \partial_s\vec X^m\cdot\vec\omega ^h\right)_{\Gamma^m}^h 
 + \frac{1}{2}\left(\mu ^{m+1, i}, \partial_s\vec X^{m+1, i}\cdot \vec\omega ^h\right)_{\Gamma ^m}^h
 - \frac{1}{2}\left(\mu ^{m+1, i+1}, \partial_s\vec X^{m+1, i}\cdot\vec\omega ^h\right)_{\Gamma ^m}^h \nonumber \\
&~~~~ - \frac{1}{2}\left(\mu ^{m+1, i}, \partial_s\vec X^{m+1, i+1}\cdot\vec \omega ^h\right)_{\Gamma ^m}^h
- \left(\mat B_q(\theta ^m)\partial_s\vec X^{m+1, i+1}, \partial_s\vec\omega ^h\right)_{\Gamma^m}^h - \varepsilon ^2\left(\partial_s\kappa ^{m+1, i+1}\vec n^m, \partial_s\vec\omega ^h\right)_{\Gamma^m}^h \nonumber \\
&~~~~ + \varepsilon ^2\left(\kappa ^{m+1, i+1}\kappa ^{m+1, i}\partial_s\vec X^{m+1, i}, \partial_s\vec\omega ^h\right)_{\Gamma ^m}^h + \varepsilon ^2\left(\frac{1}{2}\left(\kappa ^{m+1, i}\right)^2\partial_s\vec X^{m+1, i+1}, \partial_s\vec\omega ^h\right)_{\Gamma^m}^h - \varepsilon ^2\left(\left(\kappa ^{m+1, i}\right)^2\partial_s\vec X^{m+1, i}, \partial_s\vec\omega ^h\right)_{\Gamma^m}^h \nonumber \\
&~~~~-\frac{1}{\eta}\left(\frac{ x_l^{m+1, i+1} - x_l^m}{\ttau}\omega_1^h(0) + \frac{ x_r^{m+1, i+1} - x_r^m}{\ttau}\omega_1^h(1)\right) + \sigma\left(\omega_1^h(1) - \omega_1^h(0)\right) = 0, \quad \forall\vec\omega ^h\in\mathbb X^h, \\
&\left(\frac{\kappa ^{m+1, i+1} - \kappa ^m}{\ttau}, \varphi ^h\right)_{\Gamma ^m}^h - \left(\vec n^m\cdot\partial_s\left(\frac{\vec X^{m+1, i+1} - \vec X^m}{\ttau}\right), \partial_s\varphi ^h\right)_{\Gamma ^m}^h - \frac{2}{\ttau}\left(\left(\partial_s\vec X^{m+1, i}\right)^2\kappa ^{m+1, i},\varphi ^h\right)_{\Gamma^m}^h \nonumber \\
&~~~~ + \frac{2}{\ttau}\left(\partial_s\vec X^{m+1, i+1}\partial_s\vec X^{m+1, i}\kappa ^{m+1, i},\varphi ^h\right)_{\Gamma^m}^h + \frac{1}{\ttau}\left(\left(\partial_s\vec X^{m+1, i}\right)^2\kappa ^{m+1, i+1},\varphi ^h\right)_{\Gamma^m}^h + \frac{1}{\ttau}\left(\partial_s\vec X^{m+1, i}\partial_s\vec X^m\kappa ^{m+1, i}, \varphi ^h\right)_{\Gamma^m}^h \nonumber \\
&~~~~- \frac{1}{\ttau}\left(\partial_s\vec X^{m+1, i+1}\partial_s \vec X^m\kappa ^{m+1, i}, \varphi ^h\right)_{\Gamma^m}^h - \frac{1}{\ttau}\left(\partial_s\vec X^{m+1, i}\partial_s \vec X^m\kappa ^{m+1, i+1}, \varphi ^h\right)_{\Gamma ^m}^h = 0, \quad \forall\varphi ^h\in\mathbb K_0^h. 
\end{align}
\end{subequations}

The iterative process is repeated until the following convergence criteria is satisfied:
\begin{align}
    &\left\|\vec X^{m+1,i+1}-\vec X^{m+1,i}\right\|_{L^\infty}+ \left\|\mu^{m+1,i+1}-\mu^{m+1,i}\right\|_{L^\infty}+\left\|\kappa^{m+1,i+1}-\kappa^{m+1,i}\right\|_{L^\infty}\leq \text{tol}, 
\end{align}
where tol is a predefined tolerance level, ensuring that the iterative process continues until the solution reaches a specified accuracy.

\section{Numerical results}
\label{sec6}
In this section, we will perform numerical simulations based on the ES-PFEM and AC-PFEM presented above to investigate regularized solid-state dewetting of thin films in strongly anisotropic materials. 
These simulations will reveal the structure-preserving properties (energy stability, area conservation), convergence, mesh quality, and efficiency of the proposed methods. In the numerical experiments, we choose two different anisotropic energy functions: 
\begin{itemize}
    \item 2-fold anisotropy: $\gamma(\theta)=1+\beta\cos(2\theta)$;
    \item 4-fold anisotropy: $\gamma(\theta)=1+\beta\cos(4\theta)$;
\end{itemize}
where $\beta$ denotes the degree of anisotropy. For the $k$-fold anisotropy, when $\beta = 0$, it represents isotropic; when $0 \le \beta \le\frac{1}{k^2-1}$, it represents weakly anisotropic; when $\beta\ge\frac{1}{k^2-1}$, it represents strongly anisotropic. 
During all the tests, we select the tolerance tol$=1e-8$. 

\textbf{Example 1} (Convergence tests) We test convergence by quantifying the difference between surfaces enclosed by the two curves $\Gamma_1$ and $\Gamma_2$, using the manifold distance defined by 
\begin{equation*}
\text{Md}(\Gamma_1, \Gamma_2):=\left |(\Omega_1\backslash\Omega_2)\cup(\Omega_2\backslash\Omega_1) \right |=\left |\Omega_1 \right |+\left |\Omega_2 \right |-2\left |\Omega_1\cap\Omega_2 \right |, \nonumber
\end{equation*}
with $\Omega_i$, $i=1, 2$ denoting the region enclosed by $\Gamma_i$, and $| \cdot |$ representing the area of the region. Let $\vec X^m$ denote numerical approximation of surface with mesh size $h$ and time step $\ttau$, then introduce approximate solution between interval $[t_m, t_{m+1}]$ as
\begin{equation}
\vec X_{h, \ttau}(\rho, t)=\frac{t-t_m}{\ttau}\vec X^m(\rho)+\frac{t_m-t}{\ttau}\vec X^{m+1}(\rho), \quad\rho\in\mathbb{I}. 
\end{equation}
Then we further define the errors by  
\begin{equation}
e_{h, \ttau}(t)=\text{Md}(\Gamma_{h, \ttau}, \Gamma_{\frac{h}{2}, \frac{\ttau}{4}}), \quad\tilde{e}_{h, \ttau}(t)=\text{Md}(\Gamma_{h, \ttau}, \Gamma_{\frac{h}{2}, \ttau}). 
\end{equation} 
In the tests, we choose the ellipse as the initial shape, with the major axis $2$ and minor axis $1$. 
The numerical errors and orders of the ES-PFEM-method and AC-PFEM-method with 2-fold and 4-fold anisotropies are shown in Figures \ref{fig:1} and \ref{fig:2}, respectively. 
We can observe that both numerical methods exhibit an order of $O(\tau + h^2)$ for various anisotropic cases.

\begin{figure}[!htp]
\centering
\includegraphics[width=0.49\textwidth]{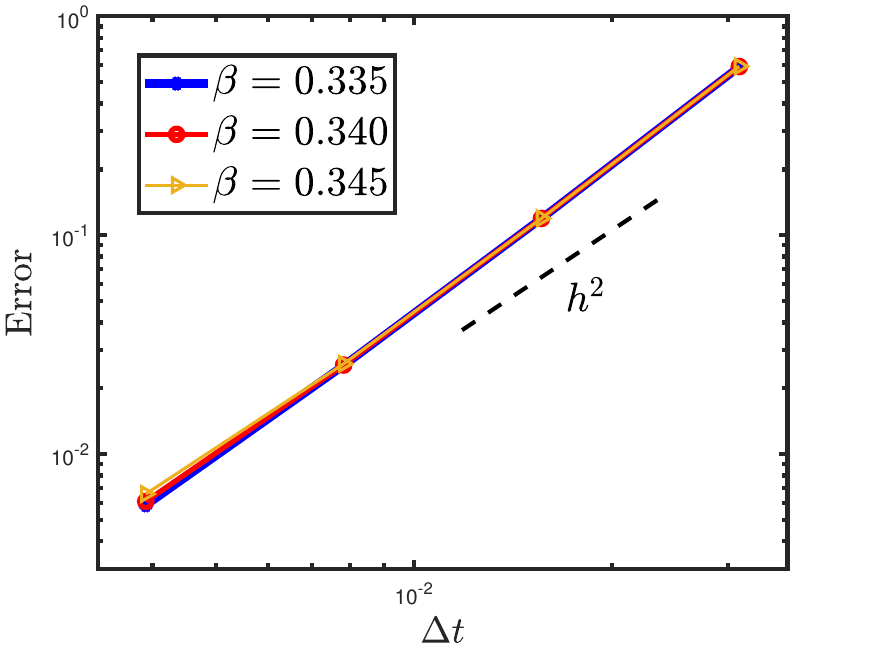}
\includegraphics[width=0.49\textwidth]{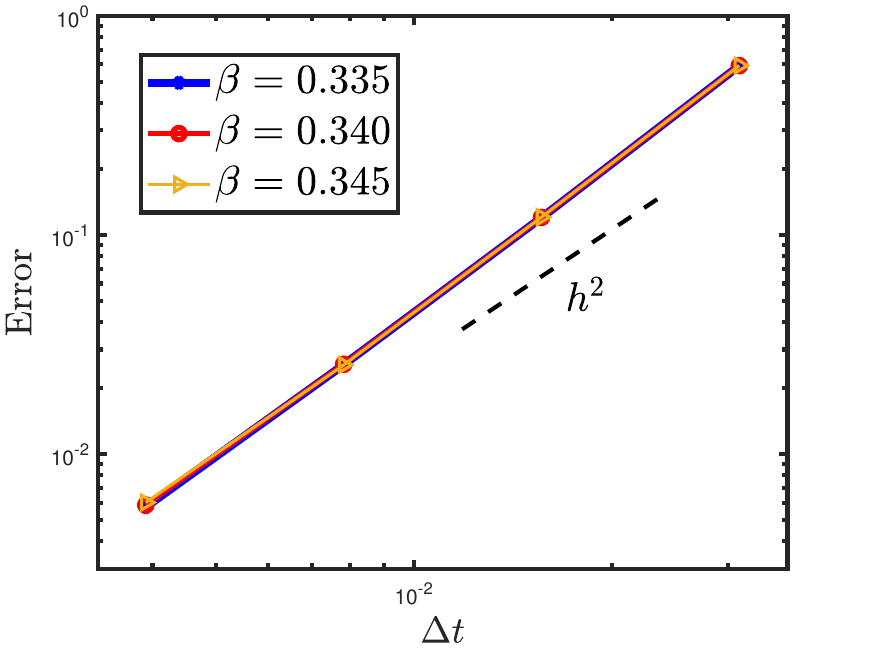}
\caption{Plot of numberical errors employing ES-PFEM at $t_m = 1$(left panel) and $t_m = 2$(right panel) for $2$-fold anisotropy. The parameters are selected as $\eta = 100$, $\sigma = -0.6$. } 
\label{fig:1}
\end{figure}

\begin{figure}[!htp]
\centering
\includegraphics[width=0.49\textwidth]{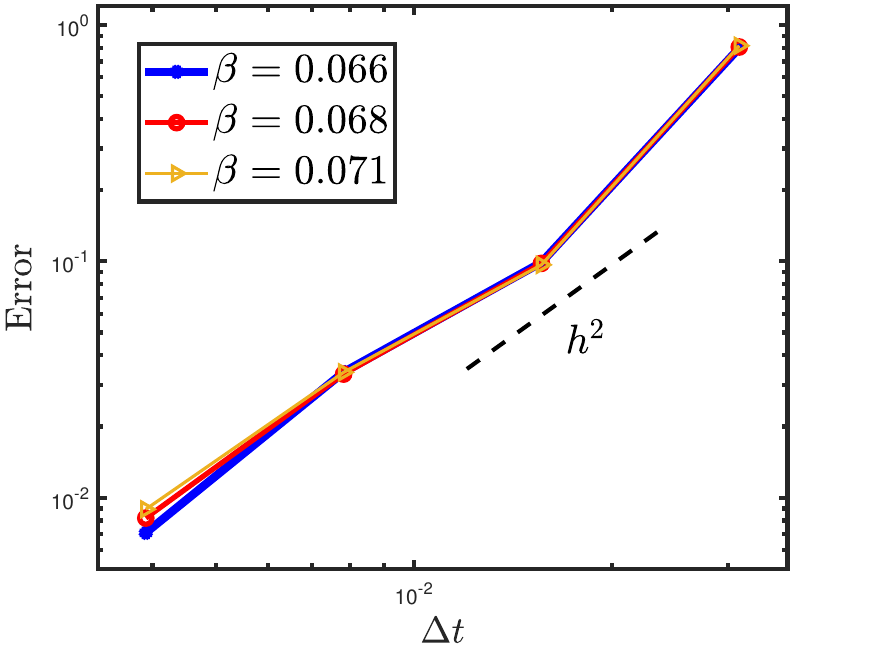}
\includegraphics[width=0.49\textwidth]{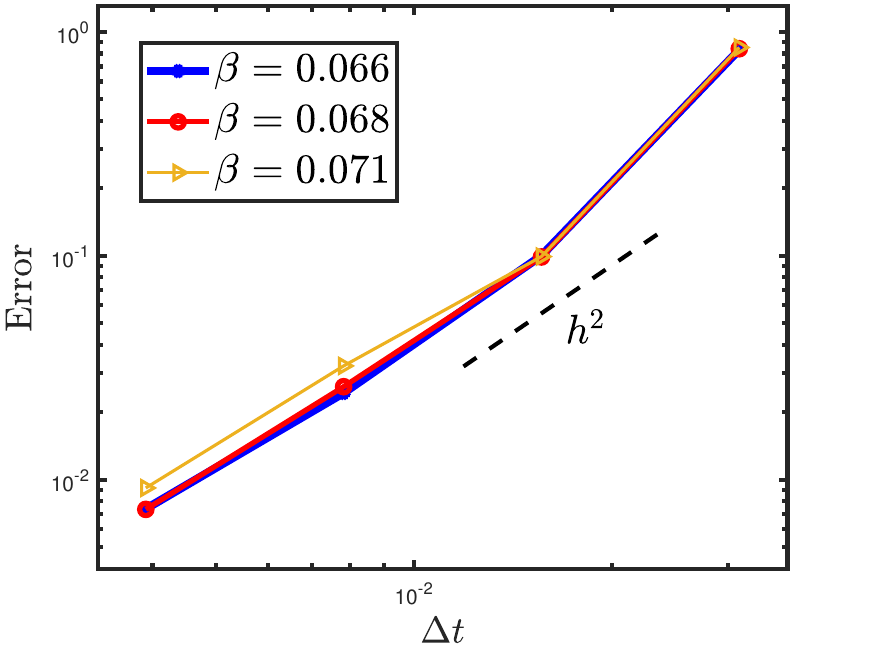}
\caption{Plot of numberical errors employing ES-PFEM at $t_m = 1$ (left panel) and $t_m = 2$ (right panel) for $4$-fold anisotropy. The parameters are selected as $\eta = 100$, $\sigma = -0.6$. }
\label{fig:2}
\end{figure}

\textbf{Example 2}: (Energy stability \& Area conservation) 
In this example, we primarily study the variation of energy during the evolution process with different types of anisotropy $\gamma (\theta)$, and the area loss using ES-PFEM and AC-PFEM methods. 
We choose the semi-ellipse as the initial data, the same as in Example 1, and the discretization parameters are set to $J = 256$ and $\tau = \frac{5}{256}$. We select contact line mobility $\eta = 100$ and cosine value of the Young contact angle
$\sigma = -0.6$. 
The numerical results are presented in Figures \ref{fig:3}-\ref{fig:5}. We can observe energy stability for both the ES-PFEM and AC-PFEM, and mass conservation for the AC-PFEM, which are consistent with our theoretical proof. As shown in Figure \ref{fig:3}, the equilibrium energy decreases as the anisotropy strengthens. In Figures \ref{fig:4}-\ref{fig:5}, besides noting the area-preserving characteristic of the AC-PFEM, we can also observe that the energy of the ES-PFEM algorithm near equilibrium is lower than that of the AS-PFEM algorithm, possibly due to the loss of area.

\begin{figure}[!htp]
\centering
\includegraphics[width=0.45\textwidth]{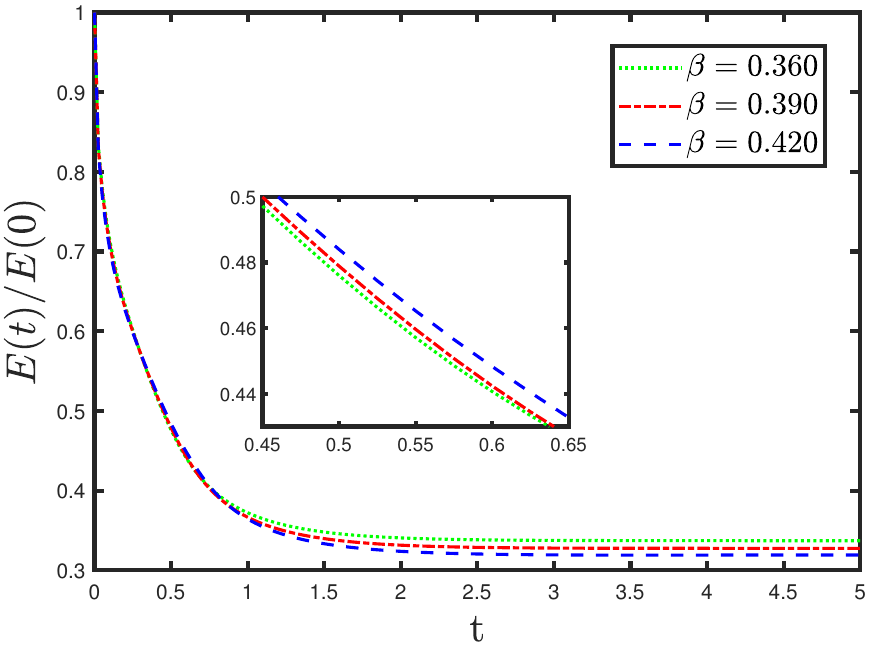}\hspace{0.5cm}
\includegraphics[width=0.45\textwidth]{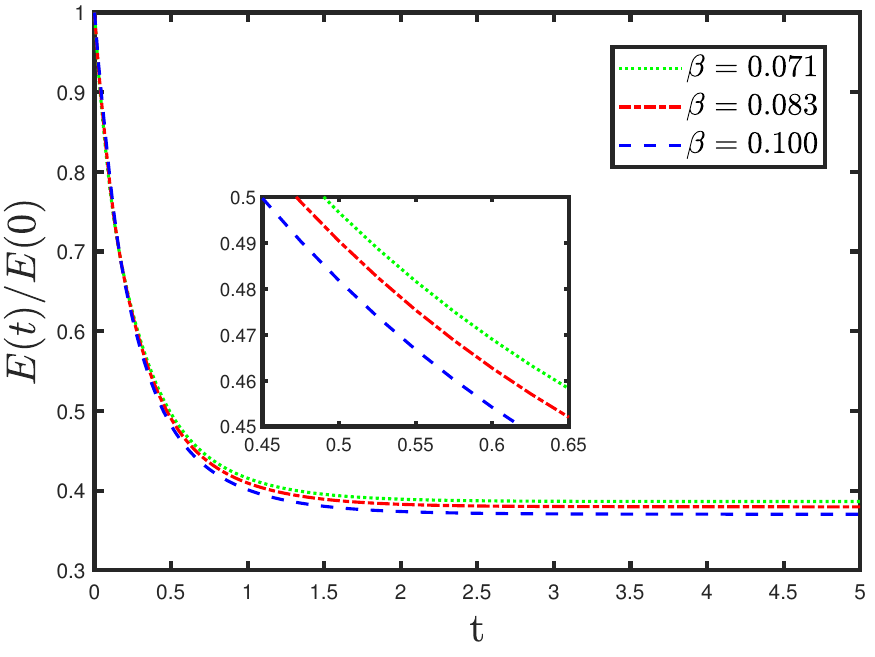}
\caption{The time history of the energy $E(t)/E(0)$ using ES-PFEM with $2$-fold (left panel) and $4$-fold (right panel) anisotropy. }
\label{fig:3}
\end{figure}

\begin{figure}[!htp]
\centering
\includegraphics[width=0.3\textwidth]{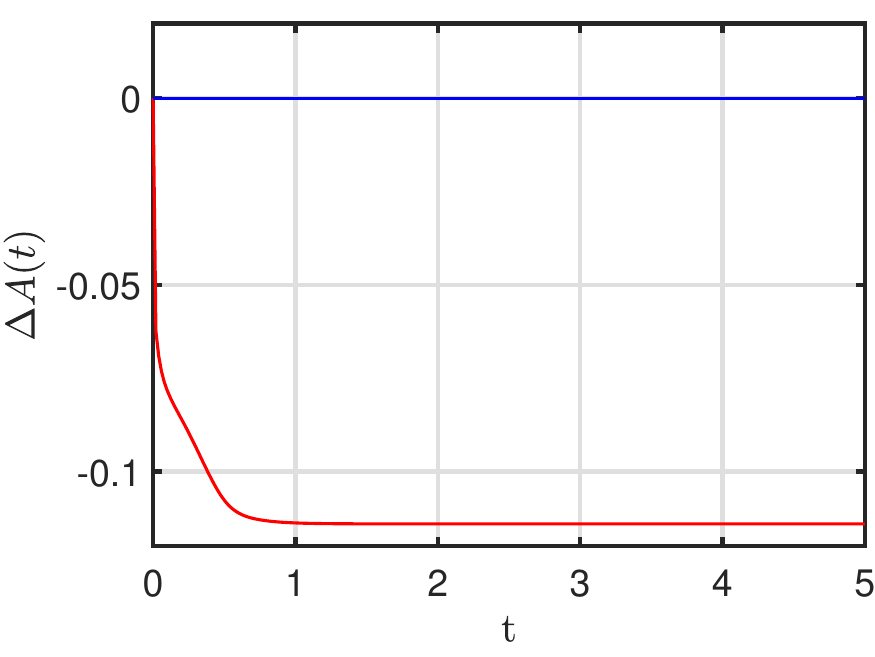}\hspace{0.5cm}
\includegraphics[width=0.3\textwidth]{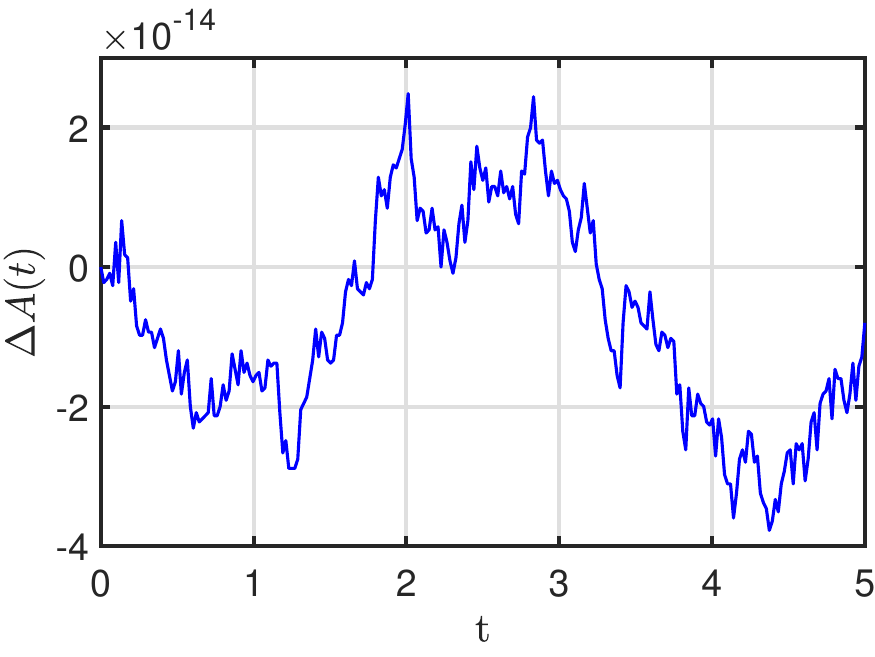}\hspace{0.5cm}
\includegraphics[width=0.3\textwidth]{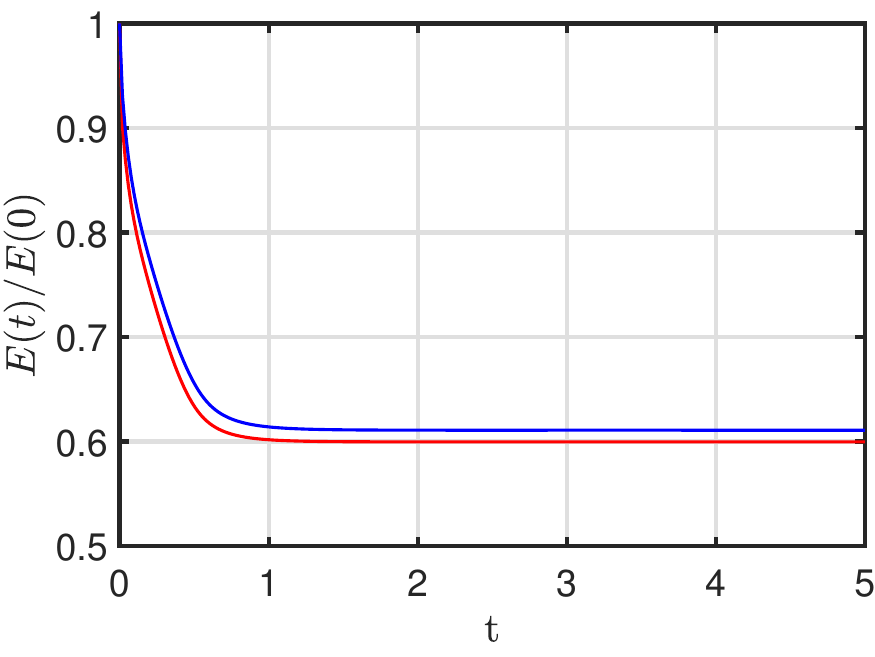}
\caption{The time history of the area loss $\triangle A(t)$ and the energy $E(t)/E(0)$. The blue line represents the results obtained using the AC-PFEM, while the red line represents the results obtained using the ES-PFEM for $2$-fold anisotropy. The degree of anisotropy is chosen as $\beta = 9/24$. }
\label{fig:4}
\end{figure}

\begin{figure}[!htp]
\centering
\includegraphics[width=0.3\textwidth]{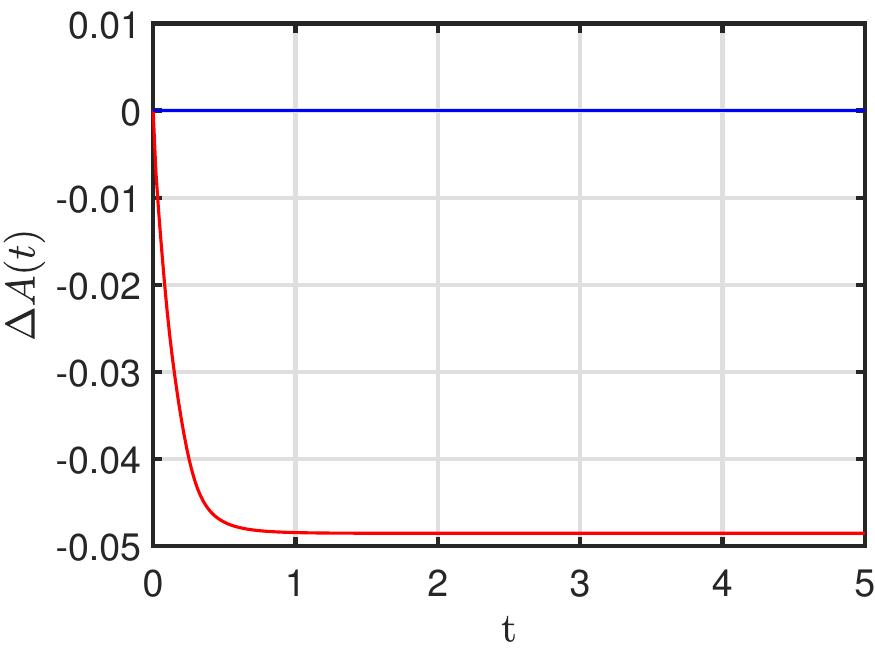}\hspace{0.5cm}
\includegraphics[width=0.3\textwidth]{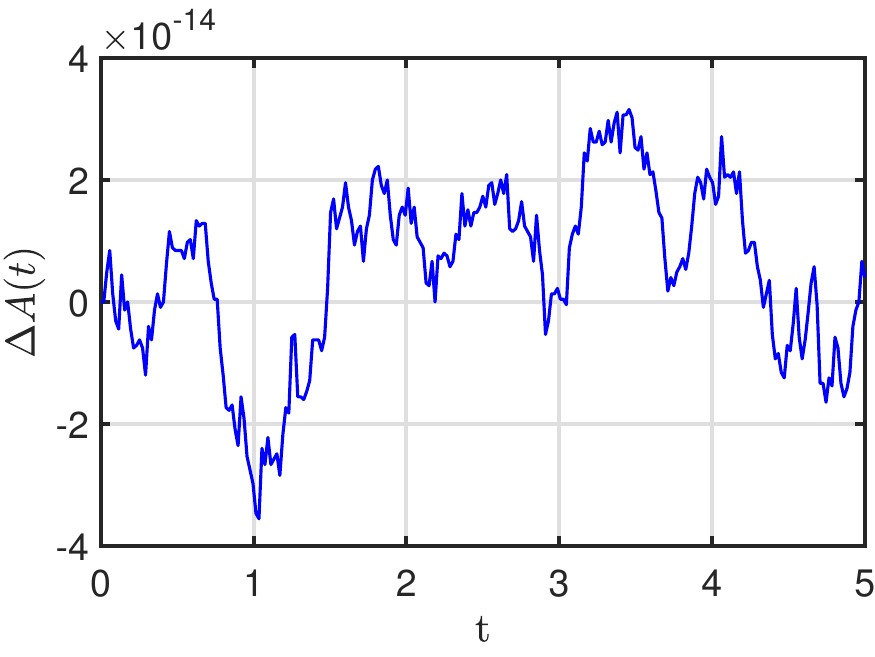}\hspace{0.5cm}
\includegraphics[width=0.3\textwidth]{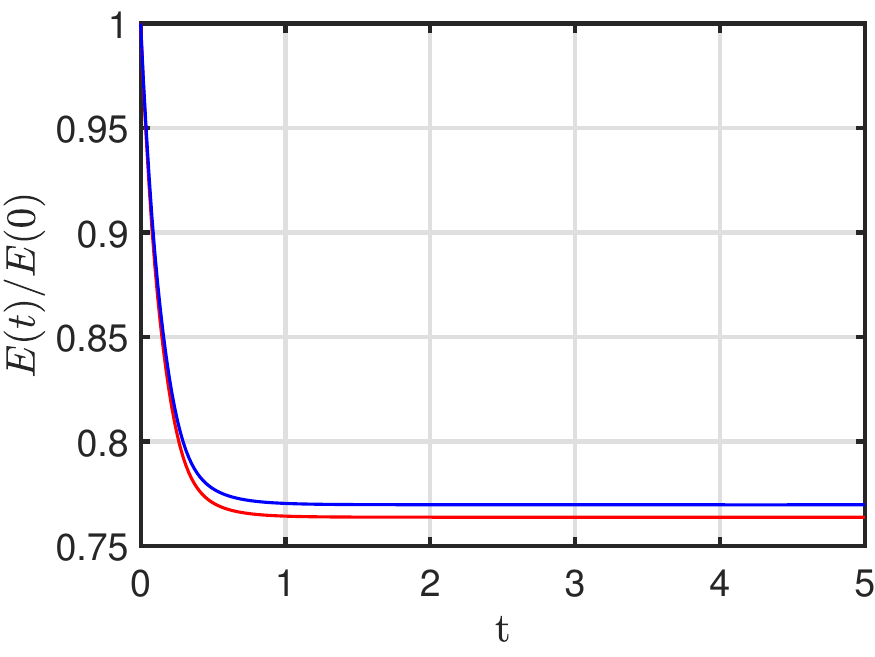}
\caption{The time history of the area loss $\triangle A(t)$ and the energy $E(t)/E(0)$. The blue line represents the results obtained using the AC-PFEM, while the red line represents the results obtained using the ES-PFEM for $4$-fold anisotropy. The degree of anisotropy is chosen as $\beta = 1/10$. }
\label{fig:5}
\end{figure}

\textbf{Example 3}: (Mesh quality)
In this example, we mainly check the mesh quality during the evolution process for the ES-PFEM. 
We choose the semi-ellipse as the initial data and and contact line mobility $\eta = 100$ and cosine value of the Young contact angle
$\sigma = -0.6$. 
For both 2-fold and 4-fold anisotropy under strongly anisotropic parameters \(\beta\), when \(\varepsilon = 0\), the ES-PFEM degenerates into the PFEM without the regularization term. As observed in Figure \ref{fig:6}, the mesh ratio of the ES-PFEM with \(\varepsilon = 0\) increases significantly over time, indicating a substantial deterioration in mesh quality.
When we slightly increase \(\varepsilon\) to \(5e-3\) and then to \(1e-2\), we can see a significant improvement in mesh quality. This indicates that adding the Willmore regularization term is very beneficial for the mesh quality during the numerical evolution process of the ES-PFEM.

Figure \ref{fig:7} compares the mesh ratios of the ES-PFEM for 2-fold weakly anisotropic solid-state dewetting models, with and without the Willmore regularization term. It can be seen that the mesh ratios of the two cases do not differ significantly. Additionally, Figure \ref{fig:8} compares the mesh ratios of the ES-PFEM for 2-fold strongly anisotropic systems with and without the Willmore regularization term. The comparison reveals that the ES-PFEM method with the Willmore regularization term can significantly improve mesh quality. Figures \ref{fig:9} and \ref{fig:10} show similar conclusions for the 4-fold anisotropy. These numerical experiments demonstrate that incorporating the Willmore regularization term significantly improves mesh quality for the numerical methods of strongly anisotropic solid-state dewetting models, underscoring its necessity.

\begin{figure}[!htp]
\centering
\includegraphics[width=0.45\textwidth]{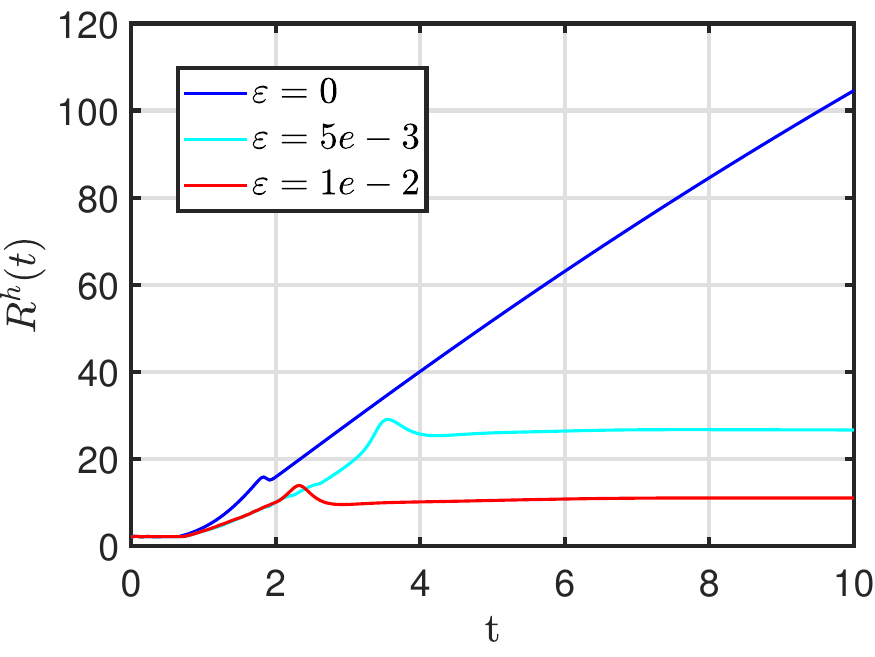}\hspace{0.5cm}
\includegraphics[width=0.45\textwidth]{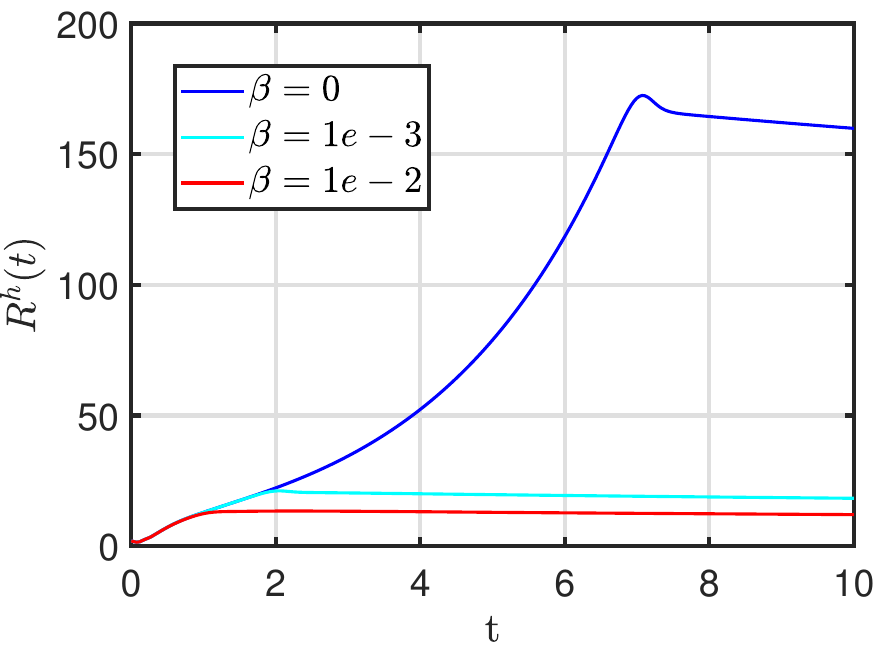}
\caption{Time evolution of the mesh ratio $R^h(t)$ for the two cases of $2$-fold (left panel) and $4$-fold (right panel).The parameters are chosen as $\ttau = 5/128$, $J = 128$, $\eta = 100$, and $\sigma = -0.6$. }
\label{fig:6}
\end{figure}

\begin{figure}[!htp]
\centering
\includegraphics[width=0.45\textwidth]{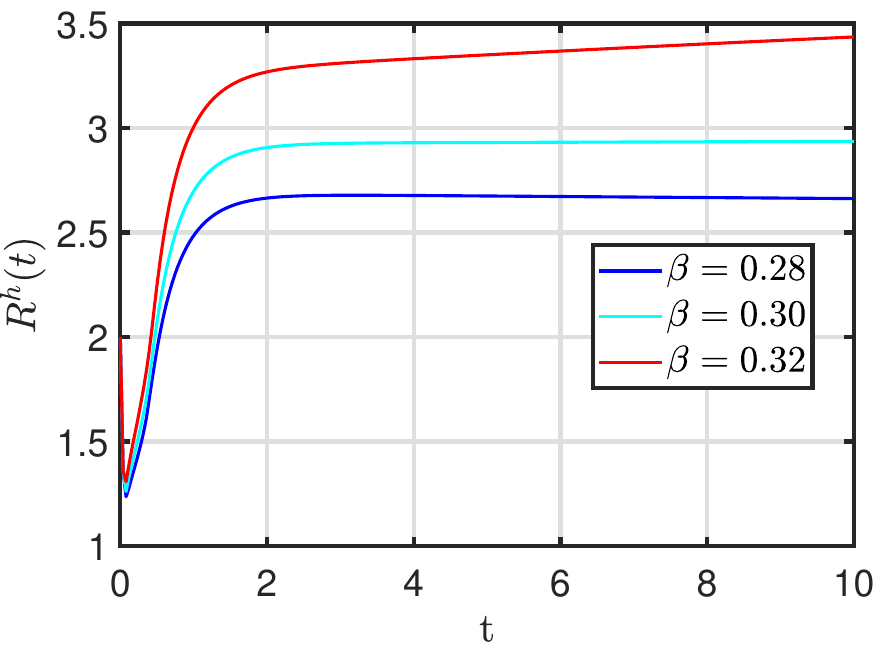}\hspace{0.5cm}
\includegraphics[width=0.45\textwidth]{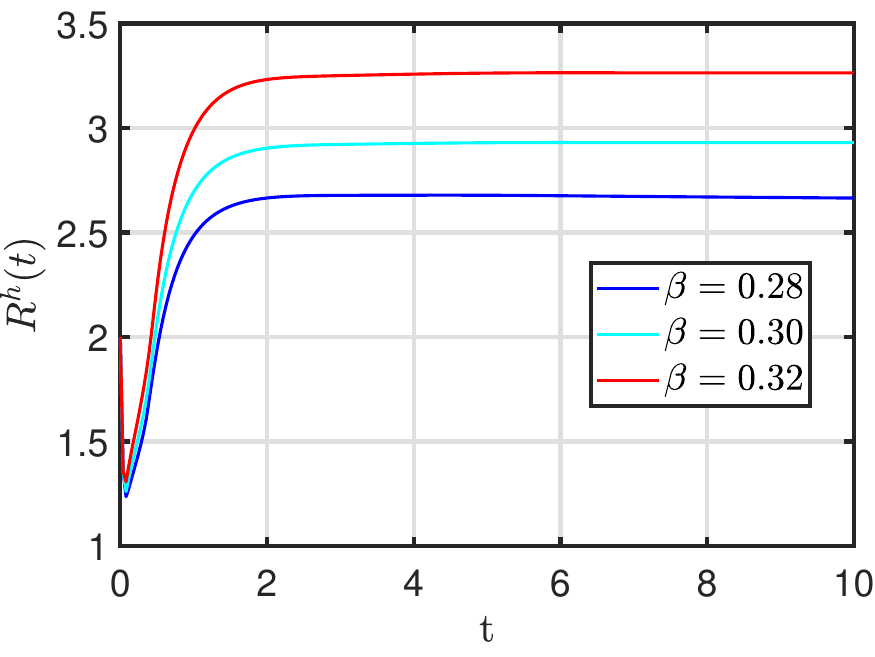}
\caption{Time evolution of the mesh ratio $R^h(t)$ without the Willmore regularization term (left panel) and with the Willmore regularization term (right panel) for $2$-fold with weak anisotropy. The other parameters are chosen as $\ttau = 5/128$, $J = 128$, $\eta = 100$, and $\sigma = -0.6$. }
\label{fig:7}
\end{figure}

\begin{figure}[!htp]
\centering
\includegraphics[width=0.45\textwidth]{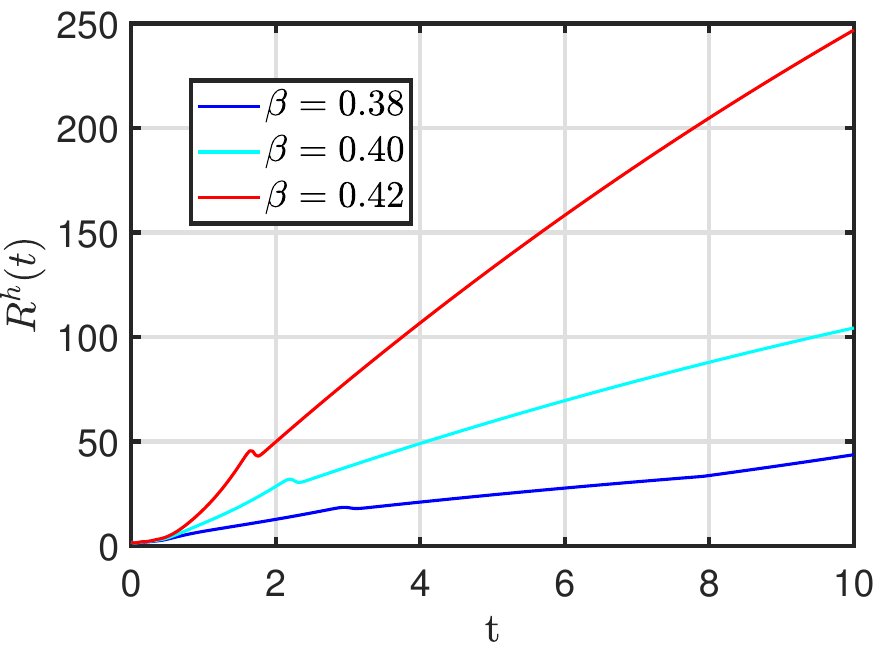}\hspace{0.5cm}
\includegraphics[width=0.45\textwidth]{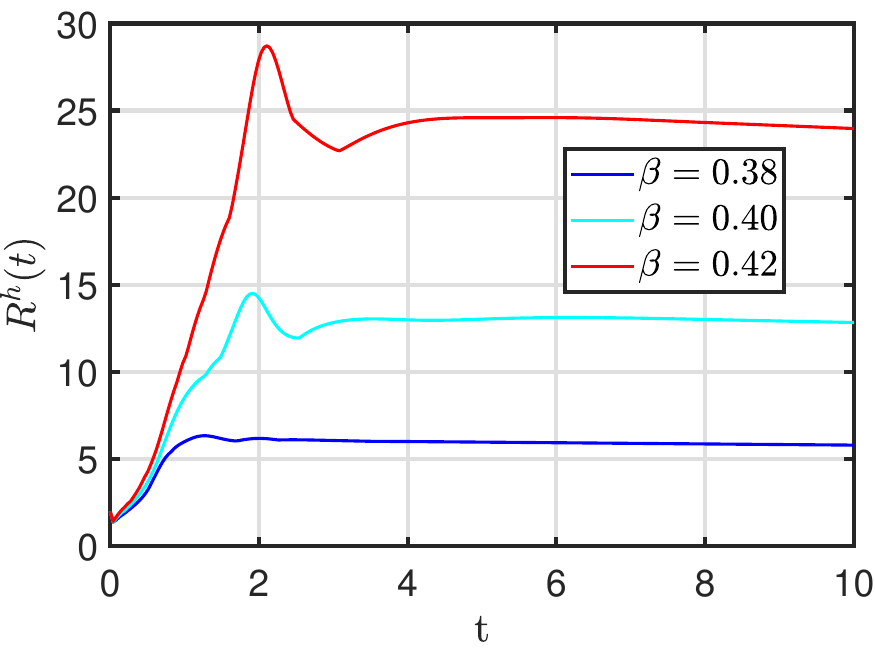}
\caption{Time evolution of the mesh ratio $R^h(t)$ without the Willmore regularization term (left panel) and with the Willmore regularization term (right panel) for $2$-fold with strong anisotropy. The other parameters are chosen as $\ttau = 5/128$, $J = 128$, $\eta = 100$, and $\sigma = -0.6$. }
\label{fig:8}
\end{figure}

\begin{figure}[!htp]
\centering
\includegraphics[width=0.45\textwidth]{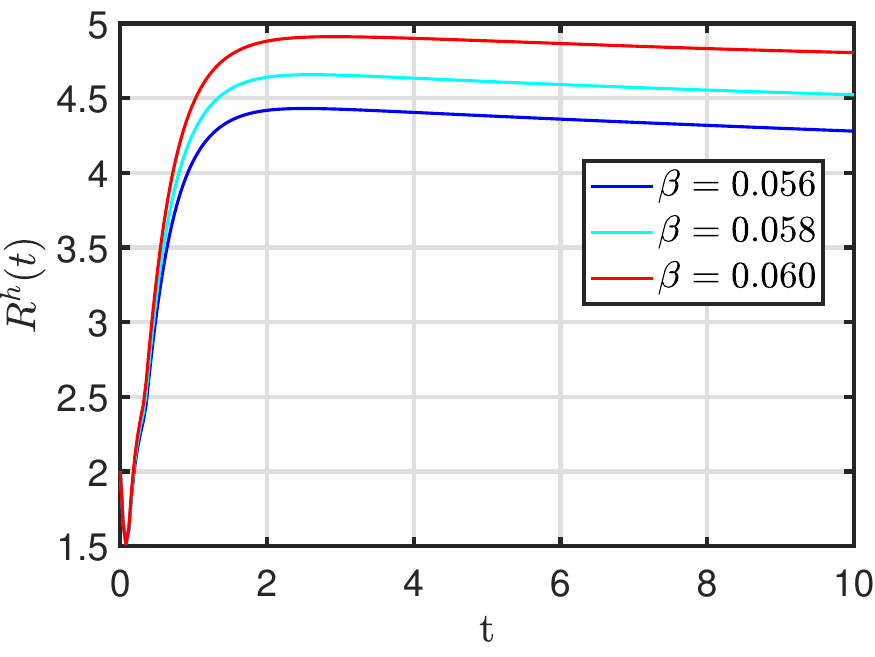}\hspace{0.5cm}
\includegraphics[width=0.45\textwidth]{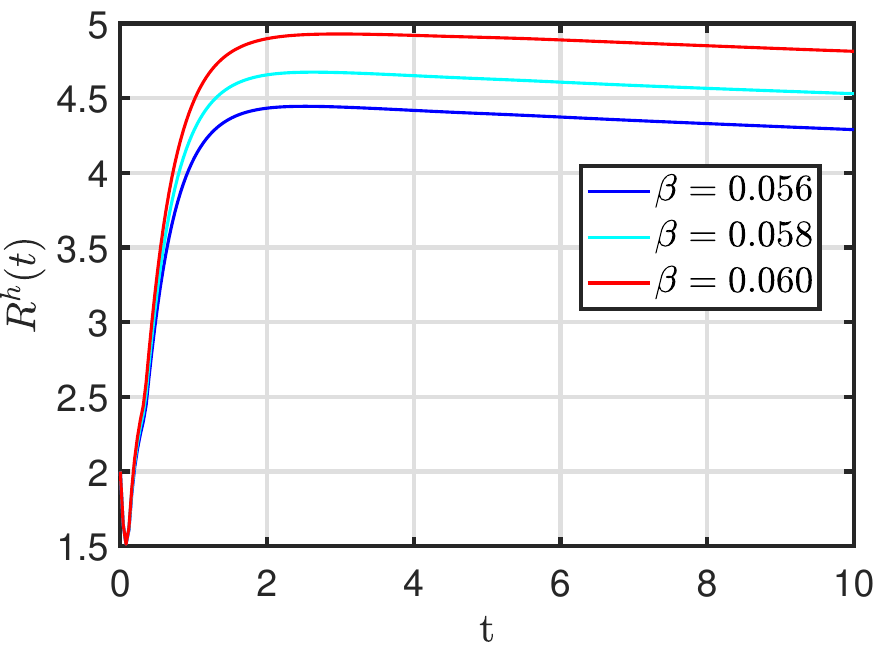}
\caption{Time evolution of the mesh ratio $R^h(t)$ without the Willmore regularization term (left panel) and with the Willmore regularization term (right panel) for $4$-fold with weak anisotropy. The other parameters are chosen as $\ttau = 5/128$, $J = 128$, $\eta = 100$, and $\sigma = -0.6$. }
\label{fig:9}
\end{figure}

\begin{figure}[!htp]
\centering
\includegraphics[width=0.45\textwidth]{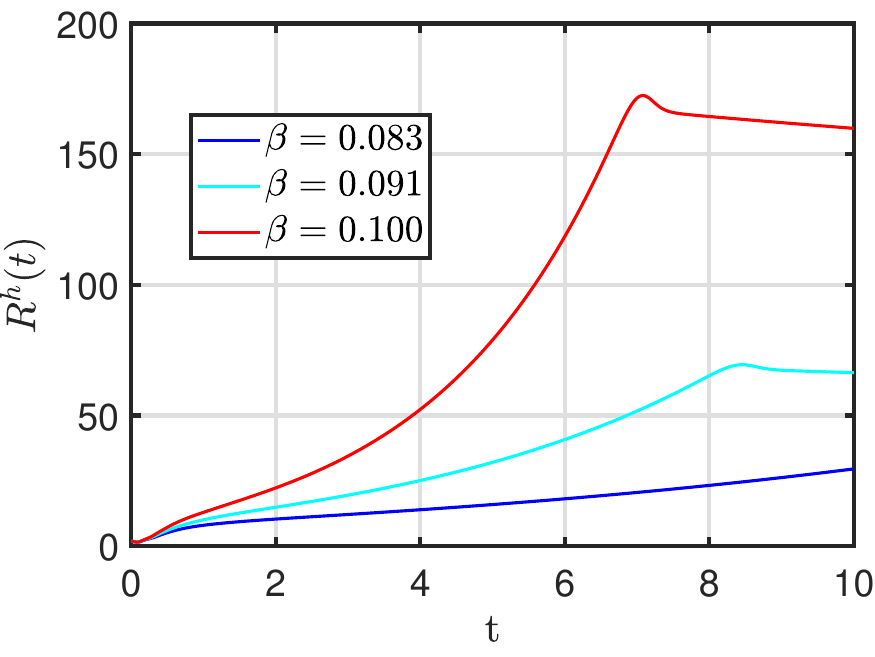}\hspace{0.5cm}
\includegraphics[width=0.45\textwidth]{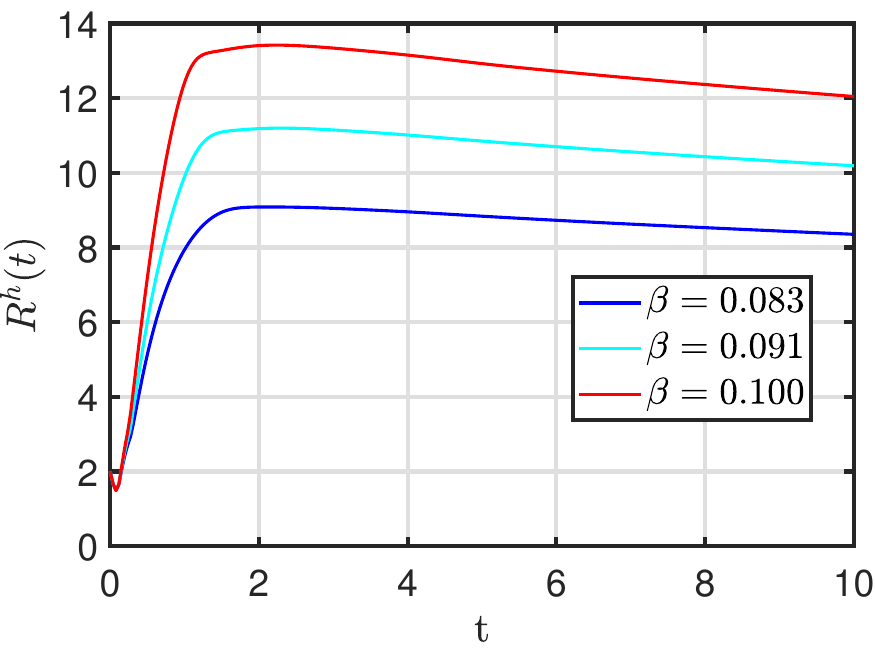}
\caption{Time evolution of the mesh ratio $R^h(t)$ without the Willmore regularization term (left panel) and with the Willmore regularization term (right panel) for $4$-fold with strong anisotropy. The other parameters are chosen as $\ttau = 5/128$, $J = 128$, $\eta = 100$, and $\sigma = -0.6$. }
\label{fig:10}
\end{figure}

\textbf{Example 4}: (Equilibrium state \& Pinch-off)
In this example, we mainly consider the intrinsic mechanisms during the evolution of the thin film and observe the evolution process of the thin film as it reaches equilibrium. 
In Figures \ref{fig:11}-\ref{fig:12}, we illustrate several evolution processes based on the ES-PFEM for the regularized system with different initial curves, tracking their progression until equilibrium is reached. We can observe that different initial curves with the same anisotropy exhibit the same equilibrium shape. It can also be observed that the thin film exhibits a pinch-off phenomenon when it is very flat. Over time, the thin film splits into several smaller films, which eventually reach equilibrium as time progresses. 

\begin{figure}[!htp]
\centering
\includegraphics[width=0.3\textwidth]{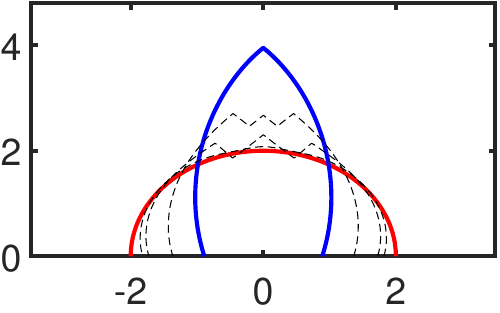}\hspace{0.5cm}
\includegraphics[width=0.3\textwidth]{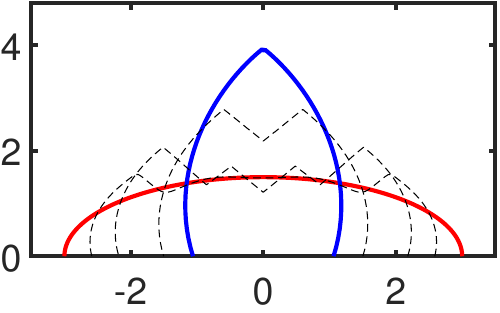}\hspace{0.5cm}
\includegraphics[width=0.3\textwidth]{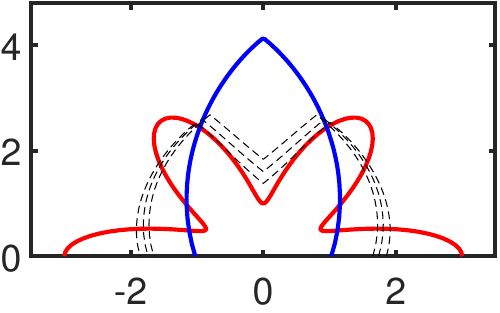}
\caption{The evolution of different initial curve (red line) to the equilibrium shape (blue line) with $2$-fold anisotropy. The degree of anisotropy are chosen as $\beta = 1/2$. The other parameters are selected as $\ttau = 1/50$, $J = 128$, $\eta = 100$, and $\sigma = -0.6$. }
\label{fig:11}
\end{figure}

\begin{figure}[!htp]
\centering
\includegraphics[width=0.3\textwidth]{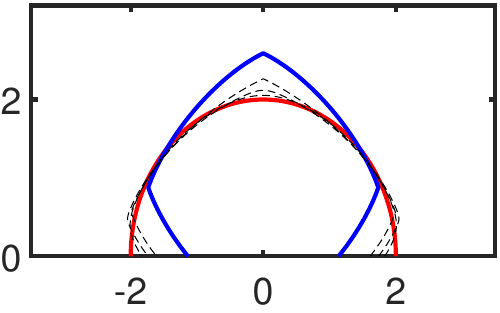}\hspace{0.5cm}
\includegraphics[width=0.3\textwidth]{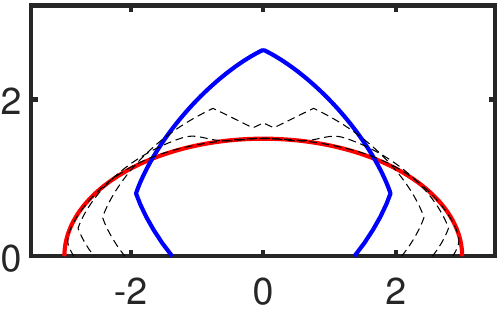}\hspace{0.5cm}
\includegraphics[width=0.3\textwidth]{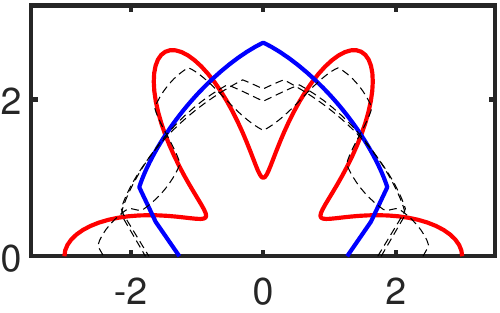}
\caption{The evolution of different initial curve (red line) to the equilibrium shape (blue line) with $4$-fold anisotropy. The degree of anisotropy are chosen as $\beta = 1/10$. The other parameters are selected as $\ttau = 1/200$, $J = 128$, $\eta = 100$, and $\sigma = -0.6$. }
\label{fig:12}
\end{figure}

\begin{figure}[!htp]
\centering
\includegraphics[width=0.3\textwidth]{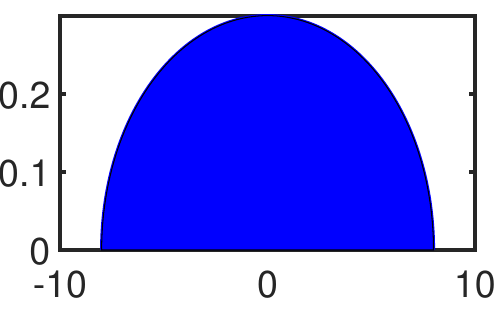}\hspace{0.5cm}
\includegraphics[width=0.3\textwidth]{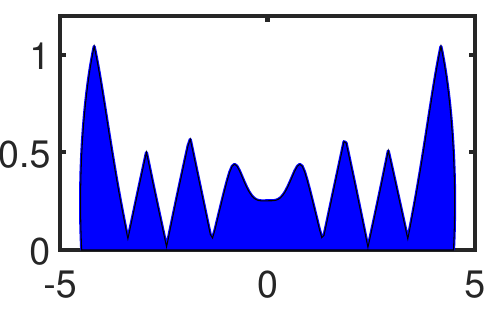}\hspace{0.5cm}
\includegraphics[width=0.3\textwidth]{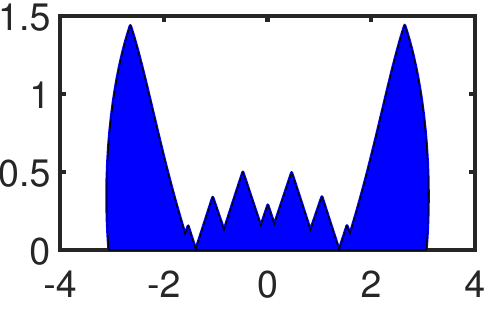}\hspace{0.5cm}
\includegraphics[width=0.3\textwidth]{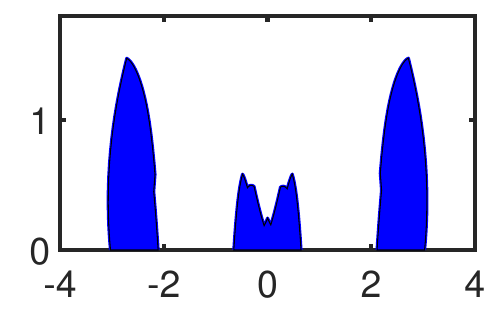}\hspace{0.5cm}
\includegraphics[width=0.3\textwidth]{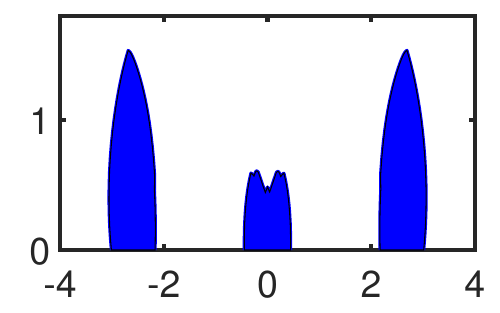}\hspace{0.5cm}
\includegraphics[width=0.3\textwidth]{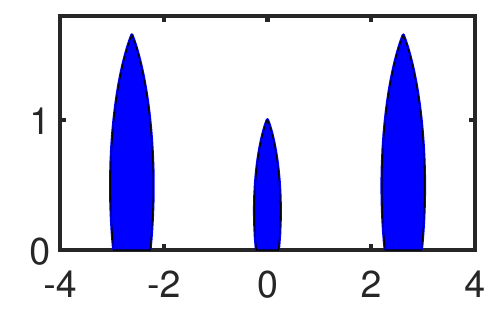}
\caption{Several snapshots in the evolution of a long, flat initial curve with $2$-fold anisotropy. The degree of anisotropy are chosen as $\beta = 1/2$. The other parameters are selected as $\ttau = 1/50$, $J = 200$, $\eta = 100$, and $\sigma = -0.6$. }
\label{fig:13}
\end{figure}

\begin{figure}[!htp]
\centering
\includegraphics[width=0.3\textwidth]{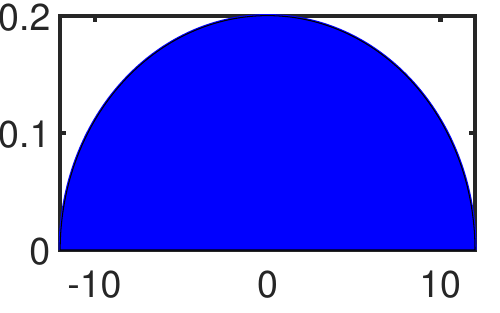}\hspace{0.5cm}
\includegraphics[width=0.3\textwidth]{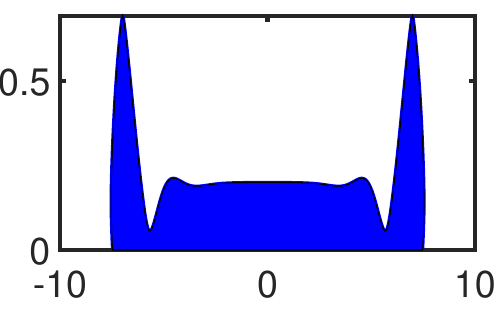}\hspace{0.5cm}
\includegraphics[width=0.3\textwidth]{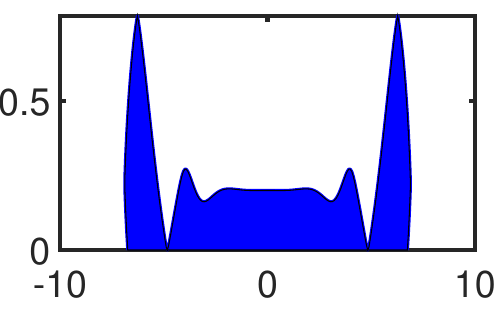}\hspace{0.5cm}
\includegraphics[width=0.3\textwidth]{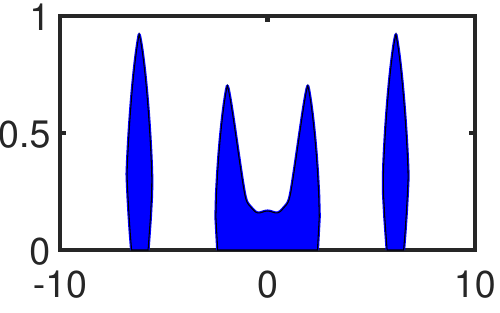}\hspace{0.5cm}
\includegraphics[width=0.3\textwidth]{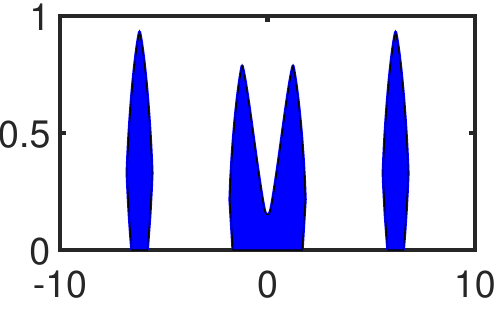}\hspace{0.5cm}
\includegraphics[width=0.3\textwidth]{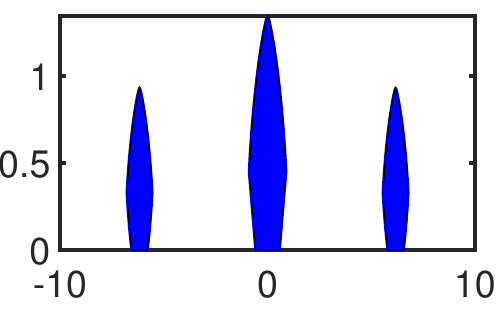}
\caption{Several snapshots in the evolution of a long, flat initial curve with $4$-fold anisotropy. The degree of anisotropy are chosen as $\beta = 1/10$. The other parameters are selected as $\ttau = 1/50$, $J = 200$, $\eta = 100$, and $\sigma = -0.6$. }
\label{fig:14}
\end{figure}
\section{Conclusions}\label{sec7}
In this work, we developed energy-stable parametric finite element approximations, including ES-PFEM and AC-PFEM, for the regularized solid-state dewetting model with strong anisotropies. 
The incorporation of the Willmore regularization term, which ensures the well-posedness of the model, was the primary motivation for constructing the regularized system considered in this work.
The key technique in our approach is the introduction of two geometric relations, inspired by recent work \cite{bao2024energy}, which allows us to establish an equivalent regularized sharp-interface model. Our detailed proof of the energy stability of the numerical scheme addresses a gap in the relevant theory.
Numerical simulations demonstrate the accuracy and efficiency of the proposed methods, revealing several advantageous properties. More importantly, extensive numerical simulations show that our schemes provide better mesh quality and are more suitable for long-term computations. 
In the future, we plan to further investigate energy-stable parametric finite element approximations for axisymmetric and three-dimensional regularized solid-state dewetting problems.

\bibliographystyle{elsarticle-num}
\bibliography{thebib}
\end{document}